\documentclass[11pt,a4paper,reqno]{amsart}
\usepackage{amssymb,amsmath,amsthm,amstext,amsfonts}
\usepackage{amsmath,amstext,amsthm,amsfonts}
\usepackage[dvips]{graphicx}
\usepackage{epic,eepic}
\usepackage[dvips]{graphicx}
\usepackage{psfrag}
\usepackage{color}
\usepackage{amssymb}
\usepackage{upgreek}
\usepackage{mathrsfs}
\usepackage{mathabx}
\makeatletter \@addtoreset{equation}{section} \makeatother

\renewcommand\thetable{\thesection.\@arabic\c@table}

\theoremstyle{plain}
\newtheorem{maintheorem}{Theorem}

\newtheorem{theorem}{Theorem }[section]
\newtheorem{proposition}[theorem]{Proposition}
\newtheorem{lemma}[theorem]{Lemma}
\newtheorem{corollary}[theorem]{Corollary}

\theoremstyle{definition} \theoremstyle{remark}
\newtheorem{remark}[theorem]{Remark}

\newtheorem{definition}[theorem]{Definition}

\newcommand{\supp}{\operatorname{supp}}

\newcommand{\al} {\alpha}

\newcommand{\de} {\delta}       

\newcommand{\vep}{\varepsilon}
\renewcommand{\epsilon}{\varepsilon}

\newcommand{\la} {\lambda}      \newcommand{\La}{\Lambda}

       \newcommand{\Si}{\Sigma}

\newcommand{\cH}{{\mathcal H}}

\newcommand{\cO}{\mathcal{O}}

\newcommand{\cN}{\mathcal{N}}

\newcommand{\cF}{\mathcal{F}}

\newcommand{\Leb}{\mbox{Leb}}

\begin{document}

\title{Positive Lyapunov exponents for Hamiltonian linear differential systems}
\author{M\'ario Bessa and Paulo Varandas}

\address{M\'ario Bessa, Departamento de Matem\'atica, Universidade da Beira Interior, Rua Marqu\^es d'\'Avila e Bolama,
  6201-001 Covilh\~a,
Portugal.}
\email{bessa@ubi.pt}

\address{Paulo Varandas, Departamento de Matem\'atica, Universidade Federal da Bahia\\
Av. Ademar de Barros s/n, 40170-110 Salvador, Brazil.}
\email{paulo.varandas@ufba.br}

\thanks{}

\date{\today}

\maketitle

\begin{abstract}
In the present paper we give a positive answer to some questions posed in \cite{Viana} on the existence of positive
Lyapunov exponents for Hamiltonian linear differential systems. We prove that there exists an open and dense set of Hamiltonian linear differential systems, over a suspension flow with bounded roof function, displaying at least one positive Lyapunov exponent. In consequence, typical cocycles over a uniformly hyperbolic flow are chaotic. Finally, we obtain similar results for cocycles over flows preserving an ergodic, hyperbolic measure with local product structure.
\end{abstract}

\section{Introduction}
\subsection{The setting}
Let $M$ denote a $d$-dimensional compact Hausdorff space $M$, $X^{t}\colon M\rightarrow M$ a continuous flow and $H\colon M\rightarrow {\mathfrak{sp}(2\ell,\mathbb{R})}$ a continuous, sometimes  smooth map,
where $\mathfrak{sp}(2\ell,\mathbb{R})$ denotes the Hamiltonian Lie algebra of
$2\ell\times 2\ell$ traceless matrices $H$ and with entries over the reals such that
$JH+H^TJ=0$, where
\begin{equation}\label{skew}
J=\begin{pmatrix}0 & -\textbf{1}_{\ell}\\\textbf{1}_{\ell} &0\end{pmatrix}
\end{equation}
denotes the skew-symmetric matrix, $\textbf{1}_{\ell}$ is the $\ell$-dimensional identity matrix and $H^T$ stands for the transpose matrix of $H$.
Given any $x\in M$, the solution $u(t)=\Phi_{H}^{t}(x)$ of the
non-autonomous linear differential equation $\partial_t
u(t)=H(X^{t}(\cdot))\cdot u(t)$, with initial condition
$\Phi_{H}^{0}(x)=\textbf{1}_{2\ell}$, is a
linear flow which evolves in the symplectic linear group
$sp(2\ell,\mathbb{R})$. The transversal linear Poincar\'{e} flow (see~\cite[\S2.3]{BD}) of a Hamiltonian flow defined in a $(2\ell+2)$-dimensional
manifold and such that $\|\partial_t X^t(x)|_{t=0}\|=\|X(x)\|\not=0$ for
all regular points $x$, like e.g. the Hamiltonian geodesic flow, is the common example of a
non-autonomous Hamiltonian linear differential system.

The framework of Hamiltonian linear differential systems is a good start if one
aims to understand the behavior of the dynamical linear differential system (see \S\ref{IM} for details) associated
to a given Hamiltonian flow. However, the linear differential systems keeps the independent relation between
the base and fiber dynamics which is a natural counterweight to its
great generality. In other words we are able to perturb the fiber
keeping unchanged the base dynamical system (or vice-versa) but on the other
hand we allow a vast number of symplectic actions in the fiber. We should keep
in mind that any perturbation in the action of the fiber of a
dynamical linear differential system should begin with a perturbation in the dynamical system
itself which, in general, cause extra difficulties.

\subsection{Lyapunov exponents}

Given a linear differential system $H$ over a flow $X^t$ the
Lyapunov exponents detect if there are any exponential asymptotic
behavior on the evolution of the time-continuous cocycle $\Phi_H^t$
along orbits (cf. \cite{BP}). If the flow is over a fixed point then $H(t)=H$ is
constant, hence the Lyapunov exponents are exactly the logarithm of the real parts of
the eigenvalues of $H$. In general, the eigenvalues of the matrix
$H(t)$ are meaningless if one aims to study the asymptotic
solutions. Under certain measure preserving assumptions on $X^t$ and
integrability of $\log\|\Phi_H^1\|$ the existence of Lyapunov exponents for almost
every point is guaranteed by Oseledets' theorem (\cite{O}).
Non-zero Laypunov exponents assure, in average, exponential
rate of divergence or convergence of two neighboring trajectories,
whereas zero exponents give us the lack of any kind of average
exponential behavior. A flow is said to be \emph{nonuniformly
hyperbolic} if its Lyapunov exponents associated to the dynamical linear differential system given by the transversal linear Poincar\'e flow, are all different from zero. 
The corresponding definitions for the discrete-time case are completely analogous. 
A central question in dynamical systems is to determine whether we
have non-zero Lyapunov exponents for the original dynamics and some
or the majority of nearby systems, an answer that usually depends on
the smoothness and richness of the dynamical system, among other
aspects.

\subsection{State of the art}

Concerning with continuous flows over compact Hausdorff spaces, and
motivated by the works by Bochi and Viana~\cite{B,BV2} for discrete
dynamical systems, the first author proved in~\cite{Be1,Be2}
that there exists a residual $\mathcal{R}$, i.e. a $C^0$-dense $G_\delta$,
such that any conservative linear differential system in
$\mathcal{R}$ yields the dichotomy: either the Oseledets
decomposition along the orbit of almost every point has a weak form
of hyperbolicity called \emph{dominated splitting} or else the spectrum is
trivial mean that all the Lyapunov exponents vanish. The main idea
behind the proof of these results, also used by Novikov~\cite{No}
and by Ma\~n\'e~\cite{M1},  is to use the absence of dominated
splitting to cause a decay of the Lyapunov exponents by perturbing
the system rotating Oseledets' subspaces thus mixing different expansion rates.

Other approaches were given in ~\cite{BeV} where it was
proved abundance of trivial spectrum but with respect to $L^p$
topologies for a large class of linear differential systems.

In this work we are interested in proving abundance of non-zero
Lyapunov exponents. In fact, a major breakthrough in the 
analysis of the Lyapunov exponents of H\"older continuous cocycles 
over nonuniformly hyperbolic base map was obtained in an outstanding paper by Viana~\cite{Viana}
and our purpose here is to contribute to the better understanding of the ergodic theory
of Hamiltonian linear differential systems and to answer some of the questions raised in that article, namely part of Problem 6 in \cite[pp. 678]{Viana}.
More precisely, we generalize to the setting
of Hamiltonian linear differential systems the results
obtained by Viana~\cite{Viana} for conservative cocycles and extended to symplectic cocycles in ~\cite{BeVar}.
Our approach is quite different from ~\cite{Viana,BeVar} mainly because the extension of many results 
concerning discrete dynamical systems
to the time-continuous setting is usually far from being immediate
and there is no direct approach to translate results from both
settings as we will now discuss.
First, it is proven that fiber-bunched cocycles admit center dynamics called holonomies.
Then, using a generalization of Ledrappier's criterium (\cite{L}), Viana proved that zero Lyapunov exponents 
correspond to a highly non-generic condition on the system $A$: conditional measures associated to invariant measures
for the cocycle are preserved under holonomies. Finally, for $sl(d,\mathbb R)$-cocycles the map 
$A\mapsto H_{A,x,y}$ is a submersion and this leads to show that the set of cocycles 
$A\in C^{r,\nu}(M, sl(d,\mathbb R))$ satisfying the later is a closed subset of empty interior.
The case of $sp(2\ell,\mathbb{R})$-symplectic cocycles have more subtleties as pointed out by Viana
~\cite[page 678]{Viana}, since the fundamental and elegant lemma asserting that
the holonomy maps are submersions as function of $A\in C^{r,\nu}(M,sp(2\ell,\mathbb K))$ fails
to be true because  the symplectic group $sp(2\ell,\mathbb{R})$ has dimension $\ell(2\ell+1)$ which is smaller 
than the necessary dimension $2\ell(2\ell-1)$, $(\ell\geq 2)$. This lead to the question of understanding which groups 
can be taken to obtain non-trivial spectrum.
In \cite{BeVar} we used a symplectic perturbative approach in small neighborhoods of heteroclinic points 
to show that every cocycle $A$ is $C^{r+\nu}$-approximated by open sets of symplectic cocycles so that unstable 
holonomies remain unchanged while stable holonomies are modified in order \emph{not} to satisfy a
rigid condition.

Here we deal with the time-continuous setting. First we address  the case of suspension flows
(as a model to flows that admit a global cross-section) and then deal with general non-uniformly hyperbolic flows.
The strategy used to prove the result for Hamiltonian linear differential systems over 
suspension flows is to make a reduction to the discrete-time case by considering an 
induced cocycle in the fiber that also depends on the roof function. 

The time-continuous results \emph{do not} follow immediately from the discrete-time ones.
Actually, in spite of perturbing the discrete-time cocycles $B \in C^{r,\nu}(M, sp(2\ell,\mathbb K))$ in 
order to make the cocycle in some sense typical situation, our perturbations (see \S\ref{PH}) are on the space of its 
\emph{infinitesimal generators} or, more accurately, on the Hamiltonian linear differential system $H \in C^{r,\nu}(M, \mathfrak{sp}(2\ell,\mathbb R))$ generating a fundamental 
solution $\Phi_H^t$ over the suspension flow $(X^t)_t$. One of the main difficulties is really to
analyze the variation of the holonomies for the reduced cocycle
$\Psi_H$ (see Section \S\ref{s.Lyap}) as a function of the infinitesimal generators $H$.
On the other hand, one could hope that it is possible to obtain a
proof of the discrete corresponding theorems for flows by reducing
to the time-one diffeomorphism $f=X^{1}$. This is known as the
\emph{embedding problem}: given a diffeomorphism can it be embedded
as the time-one map of a flow? However, it is well-known that $C^1$
diffeomorphisms that embed in $C^1$ flows form a nowhere dense set
(see~\cite{Pal}).
The general case of flows, not necessarily suspensions, is more subtle. 

As usually, the study of the hyperbolic map projected in the normal bundle and also the analysis of induced return maps is of great help. 

Our paper is organized as follows. In Section~\ref{s.statements} we present some useful definitions and state our main results.
We collect some preliminary results on hyperbolic and suspension flows and Lyapunov exponents of induced cocycles in Section~\ref{s.preliminaries}, while
the proofs of the main results are given in Sections~\ref{sec:time.continuous.suspension} and \ref{sec:time.continuous.general}. In Section~\ref{sec:time.continuous.suspension} we treat the Hamiltonian linear differential systems in the particular case when the base dynamics evolve in a suspension flow and, finally, in Section \ref{sec:time.continuous.general} we obtain the full statement for Hamiltonian linear differential systems over general flows.

\section{Some definitions and statement of the main results}\label{s.statements}

\subsection{Hamiltonian linear differential systems}\label{discrete-time}

Let $\omega$ be a symplectic form, i.e., a closed and nondegenerate 2-form. We endow the vector space $\mathbb{K}^{2\ell}$ with $\omega$ and we say that the linear map $A\colon \mathbb{K}^{2\ell}\rightarrow\mathbb{K}^{2\ell}$  
 is  \emph{symplectic} if 
$A^*\omega=\omega$, that is  $\omega(u,v)=\omega(A(u),A(v)) \text{ for all }u,v\in \mathbb{K}^{2\ell}$.
The $\ell$-times wedging 
$\omega \wedge \omega \wedge \dots \wedge \omega$ is a volume-form on $\mathbb{K}^{2\ell}$ (see e.g. \cite[Lemma 1.3]{R}). We identify the symplectic linear automorphisms with the set of matrices and denote by $sp(2\ell,\mathbb{R})$ 
($\ell\geq 1$), the group of
$2\ell\times 2\ell$ matrices $A$ and with real entries satisfying $A^TJA=J$.

Given a subspace $S\subset \mathbb{K}^{2\ell}$ we denote its $\omega$-\emph{orthogonal complement} by $S^{\perp}$ which is defined by those vectors $u\in \mathbb{K}^{2\ell}$ such that $\omega(u,v) = 0$, for all $v\in S$. Clearly $\dim(S^{\perp})=2\ell-\dim(S)$. When, for a given subspace $S\subset \mathbb{K}^{2\ell}$, we have that $\omega|_{S\times S}$ is non-degenerate (say $S^{\perp}\cap S=\{\vec0\}$), then $S$ is said to be a \emph{symplectic subspace}. We say that the basis $\{e_1,...,e_{\ell},e_{\hat{1}},...e_{\hat{\ell}}\}$ is a \emph{symplectic base} of 
$\mathbb{K}^{2\ell}$ if $\omega(e_i,e_j)=0$, for all $j\not=\hat{i}$ and $\omega(e_i,e_{\hat{i}})=1$.

Now, we begin by describing the set of time-continuous Hamiltonian linear differential systems which are also called Hamiltonian skew-product flows. 
Let $\mathfrak {sp}(2\ell,\mathbb K)$, $\ell\geq 1$, denote the symplectic Lie algebra and let 
$C^{r,\nu}(M, \mathfrak {sp}(2\ell,\mathbb K))$ denote the Banach space of $C^{r+\nu}$ linear differential systems
with values on the Lie algebra $\mathfrak {sp}(2\ell,\mathbb K)$.
Given $H \in C^{r,\nu}(M, \mathfrak {sp}(2\ell,\mathbb K))$ and a smooth flow $X^t\colon M\rightarrow M$, for each $x\in M$ we
consider the non-autonomous linear differential equation
\begin{equation}\label{lve}
\partial_t u(s)\Big|_{s=t}=H(X^{t}(x))\cdot u(t),
\end{equation}
known as \emph{linear variational equation} (or equation of first
variation). Fixing the initial condition $u(0)=\textbf{1}_{2\ell}$ the unique
solution of~\eqref{lve} is called the \emph{fundamental solution} 
related to the system $H$. The solution of
~(\ref{lve}) is a linear flow $\Phi_H^t(x)\colon \mathbb
K^{2\ell}_x\rightarrow  \mathbb K^{2\ell}_{X^{t}(x)}$ in $sp(2\ell,\mathbb{R})$ which may be seen as the
skew-product flow

\begin{equation*}
\begin{array}{cccc}
\Phi_H^t\colon  & M\times\mathbb K^{2\ell} & \longrightarrow  & M\times\mathbb K^{2\ell} \\
& (x,v) & \longrightarrow  & (X^{t}(x),\Phi^{t}_{H}(x) v).
\end{array}
\end{equation*}

Moreover, the cocycle identity
$\Phi^{t+s}_{H}(x)=\Phi^{s}_{H}(X^{t}(x))\circ\Phi_{H}^{t}(x)$ holds
for all $x\in M$ and $t,s\in\mathbb{R}$. Furthermore, the transformation $H$ satisfies
$H(x)=\partial_t\Phi^{t}_{H}(x)|_{t=0}$ for all $x\in M$ and it is referred as the \emph{infinitesimal generator}
associated to $\Phi_{H}^{t}$. It follows from the previous cocycle
identity that, for every $x \in M \;\text{and}\; t \in \mathbb R$, $(\Phi_H^t(x))^{-1}= \Phi_H^{-t}(X^t(x))$.

This coincides with the solution of the differential equation
associated to the infinitesimal generator $-H$, that is, $\partial_t
u(s)|_{s=t}=-H(X^{t}(x))\cdot u(t)$, because time is reversed.

\subsection{Multiplicative ergodic theorem}\label{MET} 

If $\mu$ is a $X^t$-invariant probability measure such that $\log\|\Phi_H^{\pm1}\|\in L^1(\mu)$ then it follows from
 Oseledets' multiplicative ergodic theorem (\cite{O}) that for $\mu$-almost every $x$ there exists a decomposition 
$\mathbb K^d= E^1_x \oplus E^2_x\oplus \dots \oplus E_{x}^{k(x)}$, called, the \emph{Oseledets splitting}, and for $1\le i\le k(x)$ 
there are well defined real numbers
$$
\lambda_i(H,X^t,x)= \lim_{t\to\pm \infty} \frac1t \log \|\Phi_H^t(x) v_i\|,
	\quad \forall v_i \in E^i_x\setminus \{\vec0\}
$$
called the \emph{Lyapunov exponents} associated to $H$, $X^t$ and $x$. It is well known that, if $\mu$ is 
ergodic, then the Lyapunov exponents are almost everywhere constant.
Since we are dealing with symplectic cocycles and $sp(2\ell,\mathbb{R})\subset sl(2\ell,\mathbb{R})$, 
this implies that $\sum_{i=1}^{k(x)} \la_i(H,X^t,x) =0$.
Notice that the spectrum of a symplectic linear transformation is symmetric with respect to the $x$-axis and to $\mathbf{S}^1$. In fact, if $\sigma\in\mathbb{C}$ is an eigenvalue with multiplicity $m$ so is $\sigma^{-1}$, $\overline{\sigma}$ and $\overline{\sigma}^{-1}$ keeping the same multiplicity (see e.g. \cite[Proposition 1.5]{R}). Therefore, since Lyapunov exponents come in pairs in the symplectic setting, then
$\lambda_i(H,X^t,x)=-\lambda_{2\ell-i+1}(H,X^t,x):=-\lambda_{\hat{i}}(H,X^t,x)$ for all $i\in\{1,...,\ell\}$.
So, not counting the multiplicity and abreviating $\lambda(H,X^t,x)=\lambda(x)$, we have the increasing set of real numbers,

$
\lambda_1(x)\geq \lambda_2(x)\geq ... \geq \lambda_{\ell}(x) \geq 0 \geq -\lambda_{\ell}(x)\geq ... \geq -\lambda_2(x)\geq -\lambda_1(x),
$
or, equivalently,
$
\lambda_1(x)\geq \lambda_2(x)\geq ... \geq \lambda_{\ell}(x) \geq 0 \geq \lambda_{\hat{\ell}}(x)\geq ... \geq \lambda_{\hat{2}}(x)\geq \lambda_{\hat{1}}(x).
$

Associated to the Lyapunov exponents we have the Oseledets splitting 
\begin{equation}\label{eq.lyap5}
\mathbb K^{2\ell}
	=E^1_x\oplus E^2_x\oplus ... \oplus E^{\ell}_x \oplus E^{\hat{\ell}}_x\oplus ... \oplus E^{\hat{2}}_x\oplus E^{\hat{1}}_x.
\end{equation}

When, for a given subspace $S\subset \mathbb K^d$, we have that $\omega|_{S\times S}$ is non-degenerate, then $S$ is said to be a \emph{symplectic subspace}. The following two basic results were proved in \cite[Section~3]{BeVar}. 
 
\begin{lemma}\label{positive}
Assume that $x$ is an Oseledets $\mu$-regular point with $2\ell$ distinct Lyapunov exponents and with Oseledets decomposition in 1-dimensional subspaces 
\begin{equation}\label{eq.lyap52}
\mathbb R^{2\ell}
=E^1_x\oplus E^2_x\oplus ... \oplus E^{\ell}_x \oplus E^{\hat{\ell}}_x\oplus ... \oplus E^{\hat{2}}_x\oplus E^{\hat{1}}_x.
\end{equation}
Then, there exists a symplectic basis $\{e_1,...,e_{\ell},e_{\hat{1}},...e_{\hat{\ell}}\}$ in the fiber over $x$ formed by 
the invariant directions given by (\ref{eq.lyap52}). Furthermore, the 2-dimensional subspace $E^i\oplus E^{\hat{i}}$
is symplectic.
\end{lemma}

\begin{lemma}\label{null}
Assume that $x$ is an Oseledets $\mu$-regular point with some zero Lyapunov exponent. Then,
the associated invariant Oseledets subspace corresponding to the zero Lyapunov exponent 
has even dimension and it is symplectic.
\end{lemma}

\subsection{The dynamical linear differential system and invariant manifold theory}\label{IM} 

Let $X^t\colon M\rightarrow M$ be a nonuniformly hyperbolic flow, that is, a flow such that all the Lyapunov exponents of its dynamical linear differential system $DX^t$ (i.e. the one given by its tangent map) are different from zero, except, of course, the flow direction which has always zero Lyapunov exponent. Assume that there exists a set $\textbf{N}\subset M$ of regular points with some negative Lyapunov exponents for $X^t$.
For a full detailed exposition about the formalism of Lyapunov exponents for flows see ~\cite{BR}.  For any $x\in \textbf{N}$ we fix $\tau_x$ such that $0<\tau_x<\min\{|\la_i(x)|: \la_i(x)<0\}$. There exist a measurable function $K\colon \textbf{N}\rightarrow (0,+\infty)$ and, given any $x\in \textbf{N}$, a \emph{local stable manifold} $W^{s}_{loc}(x)$ such that:
\begin{enumerate}
\item [(i)] $T_x W^{s}_{loc}(x)=E^s_x$, i.e. the stable subspace integrates on $W^{s}_{loc}(x)$ and
\item [(ii)] $d(X^t(x),X^t(y))\leq K_x \,e^{-\tau_x\,t}d(x,y)$, for every $y\in W^{s}_{loc}(x)$.
\end{enumerate}
Let $W^{s}(x):=\{X^{-t}(W^{s}_{loc}(X^t(x)));t>0\}$ stands for the \emph{global stable manifold} of $x$. The saturate of $W^{s}(x)$ defined by $\cup_{t\in\mathbb{R}}X^t(W^{s}(x))=\cup_{t\in\mathbb{R}}W^{s}(X^t(x))$ and denoted by $W^{ws}(x)$ is sometimes called \emph{weak stable manifold}. In an analogous way, and of course if the flow has positive Lyapunov exponents, we define \emph{local unstable}, \emph{unstable} and \emph{weak stable manifolds} of $x$, which are the stable objects for the flow $X^{-t}$ and are denoted respectively by  $W^{u}_{loc}(x)$, $W^{u}(x)$ and $W^{wu}(x)$. Like in diffeomorphisms these objects vary measurably with the point $x$ and so one can
consider \emph{hyperbolic blocks} $\cH(K,\tau)$ on which both the local invariant manifolds 
$W_{\text{loc}}^s(x)$ and $W_{\text{loc}}^u(x)$ are uniform and vary continuously with $x\in \cH(K,\tau)$.

Oftentimes it is useful to work with the \emph{linear Poincar\'e flow} of the vector field $X$. This linear flow is defined by the projection on the normal subsection of the flow $N$, formally by $P_X^t(p):=\Pi_{X^t(p)}\circ DX^t_p$ where $\Pi_{X^t(p)}$ stands for the projection in the normal fiber $N_{X^t(p)}$ at $X^t(p)$. In fact the linear Poincar\'e flow is the derivative of the Poincar\'e map of the flow $\mathcal{P}_X^t$.
The following result says that, if $X$ is nonuniformly hyperbolic, then so it is its linear Poincar\'e flow.

\begin{lemma}\label{Pesin}
Let $X^t$ be a nonuniformly hyperbolic flow associated to the measure $\mu$ with invariant decomposition $E_1^s(x)\oplus ...\oplus E^s_{i}(x)\oplus E^u_1(x)\oplus ...\oplus E^u_{j}(x)$ on a $\mu$-generic point $x$ and with associated Lyapunov exponents respectively $\lambda^s_1(x)$, ..., $\lambda^s_i(x)<0$ and $\lambda^u_1(x)$,..., $\lambda^u_j(x)>0$. The Lyapunov exponents of $P_{X}^{t}(x)$ associated to the projected decomposition $N_1^s(x)\oplus ...\oplus N^s_{i}(x)\oplus N^u_1(x)\oplus ...\oplus N^u_{j}(x)$, where $N^\sigma_k(x)=\Pi_{x} (E^\sigma_k(x))$,  are respectively $\lambda^s_1(x)$, ..., $\lambda^s_i(x)<0$ and $\lambda^u_1(x)$,..., $\lambda^u_j(x)>0$.
\end{lemma}

\begin{proof}
Let us consider any $\mu$-generic point $x$ and ${u}\in{N^{u}_k(x)}\setminus\{\vec0\}$, for some $k=1,...,j$, and denote by
$\theta_{t}(x)=\measuredangle(X(X^{t}(x)),E^{u}(X^{t}(x)))$ the angle between the flow direction $X(X^t(x))$ and the direction $E^{u}(X^{t}(x))$. Then, for some $\alpha\in{\mathbb{R}}$ and $v\in{E^{u}_k(x)}$, we have
\begin{eqnarray*}
\lim_{t\rightarrow{\pm{\infty}}}\frac{1}{t}\log{\|P^{t}_{X}(x)\cdot u\|}
	&=&\lim_{t\rightarrow{\pm{\infty}}}\frac{1}{t}\log{\|\Pi_{X^{t}(x)}}\circ{DX^{t}_{x}}\cdot(\alpha X(x)+v)\|\\
	&=&\lim_{t\rightarrow{\pm{\infty}}}\frac{1}{t}\log{\|\alpha\Pi_{X^{t}(x)}}\circ{X(X^{t}(x))}+\Pi_{X^{t}(x)}\circ{DX^{t}_{x}}\cdot v\|\\
&=&\lim_{t\rightarrow{\pm{\infty}}}\frac{1}{t}\log\left({\sin(\theta_{t}(x))\|{DX^{t}_{x}}\cdot{v}}\|\right)\\
&=&\lim_{t\rightarrow{\pm{\infty}}}\frac{1}{t}\log{\sin(\theta_{t}(x))+\lim_{t\rightarrow{\pm{\infty}}}\frac{1}{t}\log{\|{DX^{t}_{x}}\cdot{v}}\|}\\
&=&\lambda^{u}_k(x),
\end{eqnarray*}
because $\underset{t\rightarrow{\pm{\infty}}}{\lim}\frac{1}{t}\log\sin(\theta_{t}(x))=0$ which is a consequence of the sub exponential growth of the angle given by the Oseledets theorem (see e.g. \cite{BP}).
The calculus for $N^{s}_k(x)$, $k=1,...i$, is analog.
\end{proof}

It follows from Pesin's theory (see e.g.~\cite{BP}) and Lemma~\ref{Pesin} that given $x\in \textbf{N}$, thus with negative Lyapunov exponents for the Poincar\'e map $\mathcal{P}_X^t$,  and fixed $\tau_x$ such that $0<\tau_x<\min\{|\la_i(x)|: \la_i(x)<0\}$, there exist a measurable function $\tilde K\colon \textbf{N}\rightarrow (0,+\infty)$ and a \emph{local stable normal manifold} $\mathscr{N}^{s}_{loc}(x)$ such that:
\begin{enumerate}
\item [(i)] $T_x \mathscr{N}^{s}_{loc}(x)=N^s(x)$  and
\item [(ii)] $d(\mathcal{P}_X^t(x),\mathcal{P}_X^t(y))\leq \tilde K_x \,e^{-\tau_x\,t}d(x,y)$, for every $y\in \mathscr{N}^{s}_{loc}(x)$.
\end{enumerate}
We observe that, in local charts on $\mathbb{R}^d$, $\mathscr{N}^{s}_{loc}(x)$ is the intersection of $W^{ws}(x)$ with a local small normal section to the flow on $x$.

Now, we proceed to describe the local product structure for non-atomic $(X^t)_t$-invariant and ergodic probability measures $\mu$. 
For that we will recall the existence of tubular neighborhoods (see e.g. \cite{AM,PdM}). Given a regular point $x\in M$ for the vector field $X$ the \emph{tubular neighborhood theorem} asserts the existence of a positive $\delta=\delta_x>0$, an open neighborhood $U^\delta_x$ of $x$, and a diffeomorphism $\Psi_x: U^\delta_x \to (-\delta,\delta)\times B(x,\delta) \subset \mathbb R \times \mathbb R^{d}$,  $B(x,\delta)$ is identified with the ball $B(\vec0,\delta)\cap\langle(1,0,...,0)^\perp\rangle$, where $\langle(1,0,...,0)^\perp\rangle$ stands for the hyperspace perpendicular to the vector $(1,0...,0)$,  such that the vector field $X$ on $U^\delta_x$ is the pull-back of the vector field $Y:= (1,0, \dots, 0)$ on $(-\delta,\delta)\times B(x,\delta)$. More precisely,  $Y=(\Psi_x)_* X:=D(\Psi_x)_{\Psi_x^{-1}}X(\Psi_x^{-1})$. In this case the associated flows are conjugated, that is, $Y^t(\cdot)=\Psi_x(X^t(\Psi_x^{-1}(\cdot)))$ for $t$ small enough.

So, given $x\in \cH(K,\tau)$ and $\epsilon>0$ small enough the size of both invariant manifolds $W^u_{\text{loc}}(X^{t}(y))$ and $W^u_{\text{loc}}(X^{t}(y))$  have size at least $\epsilon$ for all $y\in \cH(K,\tau) \cap U_x^\delta$ 
and all $t$ so that $X^t(y)\in U_x^\delta$.
In consequence, if one considers the section $\Sigma_x=\Psi_x^{-1}(\{0\}\times B(x,\delta))$ through the point $x$ 
then for any $y\in \cH(K,\tau) \cap U_x^\delta$ the intersection $\cF_y^s=W^{ws}_{loc}(y)\cap \Sigma_x$ 
(respectively $\cF_y^u=W^{wu}_{loc}(y)\cap \Sigma_x$) defines a smooth and long stable (respectively unstable) leaf on $\Sigma_x$.

Since the angles of strong stable and unstable foliations are bounded away from zero on hyperbolic blocks it is not hard to check that for all $y,z\in \cH(K,\tau) \cap U_x^\delta$ the intersection $[y,z]_{\Sigma_x}:=\cF^{u}_y \pitchfork \cF^{s}_{z}$ consists of a unique point, provided that $\delta$ is small. Set 
$$
\cN_x^u(\delta)=\{ [x,y]_{\Sigma_x} \in  \cF^u_x : y \in \cH(K,\tau) \cap U_x^\delta \} \subset \Sigma_x
$$ 
to be a $u$-neighborhood of $x$ in $\Sigma_x$ and 
$$
\cN_x^s(\delta)=\{ [y,x]_{\Sigma_x} \in  \cF^s_x : y \in \cH(K,\tau) \cap U_x^\delta \} \subset \Sigma_x
$$
an $s$-neighborhood of $x$ in $\Sigma_x$.
Set $\cN_\de(x):=\cH(K,\tau) \cap U_x^\delta$ to be a neighborhood of $x$ in $\cH(K,\tau)$. Then the 
map 

\begin{equation*}
\begin{array}{cccc}
\Upsilon_x\colon  & \cN_\de(x) & \longrightarrow  & \cN_x^u(\delta) \times \cN_x^s(\delta)\times (-\delta,\delta) \\
& y & \longrightarrow  & ( [x,y]_\Sigma , [y,x]_\Sigma, t(y)),
\end{array}
\end{equation*}
with $X^{t(y)}(y)\in\Sigma_x$ is a homeomorphism. 
We can now define local product structure for flows.

\begin{definition}\label{def:lpsf}
A hyperbolic measure $\mu$ has \emph{local product structure} if for every $x\in \supp(\mu)$ and a small $\de>0$ the measure $\mu\!\mid_{\cN_x(\de)}$ is equivalent to the product measure $\mu^u_x \times \mu^s_x\times Leb$, where
$\mu_x^{i}$ denotes the conditional measure of $(\Upsilon_x)_*(\mu\!\mid_{\cN_x(\de)})$ on  $\cN^{i}_x(\de)$, for $i\in{u,s}$ and $Leb$ stands for the 1-dimensional Lebesgue measure along the flow direction. We denote by
$\mu_\Sigma$ the marginal measure of $\mu\!\mid_{\cN_x(\de)}$ on $\Sigma$ obtained via projection along the flow direction.
\end{definition}

\begin{remark}
Our previous computations show that hyperbolic blocks for flows locally contain pieces of orbits.
Moreover, it will be important to notice that the previous property implies that, up to reduce $B(x,\vep)$, it will be independent of the size of the tube by the long tubular neighborhood theorem.
\end{remark}

\begin{remark}
Our definition does not depend on the smooth cross section through the point. In particular we may assume without loss of generality that local strong stable and unstable manifolds through $x$ are contained in $\Sigma_x$ and consequently $\cF^s_x\subset W^s_{loc}(x)$
and $\cF^u_x\subset W^u_{loc}(x)$.
\end{remark}

Let us mention that despite the fact that it seems a strong condition, the local product structure property is satisfied by 
all equilibrium states associated to H\"older continuous potentials and Axiom A flows
(see~\cite{BR75}) and by suspension (semi)flows of maps with a probability measure satisfying the corresponding local product structure of 
Definition 2.1 in \cite{BeVar}.

\subsection{Statement of the results}

The following result answered positively to problem 4 in \cite{Viana}.

\begin{theorem}(\cite[Theorem A]{BeVar})\label{Viana1}
Let $M$ be a compact Riemannian manifold. Take $f\in \mbox{Diff}^{~1+\alpha}(M)$ $(\alpha >0)$
and an $f$-invariant, ergodic and hyperbolic probability measure $\mu$ with local product structure. 
Then, there exists  an open and dense set of maps $\mathscr{O}$ in $C^{r,\nu}(M,sp(2\ell,\mathbb K))$ 
such that for any $A\in\mathscr{O}$ the cocycle $F_A$ has at least one positive Lyapunov exponent 
at $\mu$-almost every point. Moreover, the complement is a set with infinite codimension.
\end{theorem}

The following results are related to Problem 6 of \cite{Viana}
and are a continuous-time counterpart of the previous theorem.

\begin{maintheorem}\label{thm:flow}
Let $M$ be a compact Riemmanian manifold, $X^t$ be a $C^{1+\alpha}$-flow on $M$ preserving
a probability measure $\mu$. Assume that $\mu$ is hyperbolic and has local product structure. 
Then, there exists  an open and dense set of maps $\mathscr{O}$ in $C^{r,\nu}(M,\mathfrak{sp}(2\ell,\mathbb K))$ 
such that for any $H\in\mathscr{O}$ the cocycle $\Phi_H^t$ has at least one positive Lyapunov exponent 
at $\mu$-almost every point. Moreover, the complement is a set with infinite codimension.
\end{maintheorem}

It follows from the work of Bowen, Ruelle \cite{BR75} that every $C^{1+\alpha}$ uniformly hyperbolic flow is 
(semi)con\-ju\-ga\-ted to a suspension semiflow over a subshift of finite type with an H\"older continuous roof function
and that every equilibrium state associated to an H\"older continuous potential has the local product structure. Hence we obtain the following:

\begin{corollary}\label{cor:hypflow}
Given a $C^{1+\alpha}$ ($\alpha>0$) flow $X^t\colon M\rightarrow M$, $\La$ a hyperbolic set and $\mu$ an equilibrium state for $X^t\mid_{\La}$ with respect to an H\"older continuous potential: there is an open and 
dense set of (fiber-bunched) maps $\mathscr{O}$ in $C^{r,\nu}(M,\mathfrak{sp}(2\ell,\mathbb K))$ such that, if $H\in\mathscr{O}$, then the cocycle $\Phi^t_H$ has at least one positive Lyapunov exponent at $\mu$-almost every point. Moreover, their complement is a set with infinite codimension.
\end{corollary}

With respect to Theorem~\ref{thm:flow} our assumptions allow us to deal with suspension flows of
nonuniformly hyperbolic maps as Benedicks-Carleson quadratic maps, Manneville-Pommeau transformations, 
H\'enon maps or Viana maps, and also with the geometric Lorenz and Rovella-like attractors. Thus,
the majority of H\"older continuous cocycles over the flows discussed above have at least one positive 
Lyapunov exponent.  Let us mention
that Fanaae~\cite{Fan} proved that an open and dense set of \emph{fiber-bunched} 
$sl(d,\mathbb K)$-cocycles over Lorenz flows have simple spectrum. Although this result should probably extend
to arbitrary suspension flows with hyperbolic  countable Markov structure, the fiber-bunching assumption 
is crucial in the argument. 

In comparison  one obtains (via perturbative techniques on the space of infinitesimal generators) the existence 
of at least one positive Lyapunov exponent  in the sympectic case $\mathfrak{sp}(2\ell,\mathbb K)$ under 
weak assumptions,  and our theorems are stated with respect to open
and dense set of infinitesimal generators while in \cite{Fan} the author uses a stronger topology on the space of
linear differential systems with a strong domination condition. In particular, our results generalize the 
ones of \cite{Fan}  for H\"older cocycles with values in $sl(2,\mathbb K)=sp(2,\mathbb K)$ over the Lorenz flow.

\section{Preliminaries}\label{s.preliminaries}


\subsection{Hyperbolic and suspension flows}\label{sec:suspension}

In this subsection we recall some preliminaries on suspension (semi)flows
and discuss the local product structure for invariant measures. 

\subsubsection{Definition}

Assume that $M_0$ is a compact metric space and that $f: M_0 \to M_0$ is measurable. Given an 
$f$-invariant probability measure $\mu$ and a nonzero \emph{roof (or ceiling) function} $\varrho \colon M_0 \to [0, +\infty)$ 
satisfying $\varrho \in L^1(\mu)$ we define the \emph{suspension semiflow}  $(X^t)_{t\ge 0}$ over $f$ 
as given by  $X^t(x,s)=(x,s+t)$ acting on the space 
$$
M = \{ (x,t) \in M_0\times \mathbb R_0^+ :  0 \le t \le \varrho(x)\} / \sim,
$$
where $(x,\varrho(x))\sim (f(x),0)$.  In these coordinates $(X^t)_t$ coincides
with the flow consisting in the displacement along the ``vertical" direction.
If $f$ is invertible it is not difficult to check that $(X^t)_t$ defines indeed a flow, moreover $(X^t)_t$ preserves the probability measure 
$
\hat\mu=(\mu \times \Leb) / \int \varrho \, d\mu.
$
Furthermore, observe that $\hat\mu$ is uniquely defined by the previous expression as long 
as the roof function $\varrho$ is bounded away from zero.

\subsubsection{Hyperbolic flows}

Let $M$ be a closed Riemannian
manifold and $X^t\colon M \to M$ a smooth flow. Let also
$\Lambda\subseteq M$ be a compact and $X^t$-invariant set. We say that
a flow $X^t\colon \La \to \La$
is  \emph{uniformly hyperbolic} if there exists a $DX^t$-invariant
and continuous splitting $T_\La N= E^-\oplus X \oplus E^+$ and
constants $C>0$ and $0<\theta_1<1$ such that
$$
\|DX^t \mid E^-\| \leq C \theta_1^t
    \quad\text{and}\quad
\|(DX^t)^{-1} \mid E^+\| \leq C \theta_1^t,
    \; \forall t \geq 0
$$
for every $x \in M$. Uniformly hyperbolic flows have been well
studied since the 1970's and, in particular, their geometric
structure is very well understood. It follows from the work
of Bowen, Sinai and Ruelle~\cite{BR75,Bow,Si68} that hyperbolic flows admit finite
Markov partitions and that are semi-conjugated to suspension
flows over a hyperbolic map: 
there exists a subshift of finite type $T:\Si \to \Si$, an
H\"older continuous roof function $\varrho: \Si \to \mathbb
R$ and an H\"older continuous surjective transformation $p: \La(T,\tau) \to M$
such that $(X^t)_t$ is semi-conjugated to the suspension flow
$G^t:\La(T,\tau) \to \La(T,\tau)$ given by $G^t(x,s)=(x,t+s)$, where
$
  \La(T,r)=\{(x,t): x \in \Si, \; 0\leq t \le \varrho(x)\}/\sim,
$
and $\sim$ is an equivalence relation that identifies the pairs
$(x,\varrho(x))$ and $(T(x),0)$. In fact, $X^t \circ p = p \circ G^t$ for every $t$,
the cubes $C_i=p((x,t): x\in [0;i], 0\leq t \leq \varrho(x))$ are proper sets, and 
$p$ is injective in a residual subset with full measure with respect to any 
open invariant probability measure.
We say that the flow $(X^t)_t$ exhibits a Markov partition
$\mathscr{R}$ whose rectangles are given by $R_i=p([0;i]\times\{\vec0\})$.
We refer the reader to Chernov's expository paper~\cite{Chernov} 
for a detailed explanation of the theory.

In addition, it follows from the thermodynamical formalism for
hyperbolic flows established in the mid 1970's (see ~\cite{BR75})
that there exists a unique equilibrium state $\mu_\varphi$ with
respect to any fixed H\"older continuous potential $\varphi:\La \to
\mathbb R$. Moreover, $\mu_\varphi$ is obtained as a suspension of a
$T$-invariant measure $\mu_T$ with the usual local product
structure (see e.g. Definition 2.1 in \cite{BeVar}).

\subsection{Regularity and Lyapunov exponents of induced cocycles}\label{s.Lyap}

In this subsection we describe the topology in the class of Hamiltonian linear differential systems 
 and recall both the regularity and the Lyapunov spectrum of the time-$t$ map 
 associated to the solution of the Hamiltonian linear system.
We endow $C^{r,\nu}(M, \mathfrak{sp} (2\ell,\mathbb K))$ with the
$C^{r,\nu}$-topology defined using the norm 
$$
\|H\|_{r,\nu}
	=\underset{0\leq j \leq r}{\sup} \; \underset{x\in M}{\sup} \|D^j H(x)\|
	+\underset{x\not=y}{\sup}\frac{\|H(x)-H(y)\|}{\|x-y\|^\nu},$$
where $H\in C^{r,\nu}(M, \mathfrak{sp} (2\ell,\mathbb K))$ and $x,y\in M$. Let us also mention that for the proofs it is enough to consider the case
when $\nu=1$ (i.e. $H$ Lipschitz). In fact,
if $H$ is $\nu$-H\"older continuous with respect to the metric
$d(\cdot,\cdot)$ then it is Lipschitz with respect to the metric
$d(\cdot,\cdot)^\nu$. Hence, up to a change of metric we may assume
that $H$ is Lipschitz and we will do so throughout the paper.

Our starting point is to obtain that the cocycle $\Phi^t_H(x)$ is also Lipschitz continuous.

\begin{lemma}\label{lem:timetau}
Given any $t\in\mathbb R$ and $H\in C^{r,\nu}(M, \mathfrak{sp} (2\ell,\mathbb K))$ there exists $C_1=C_1(t,H)>0$ such that, for all $y,z\in M$, we have 
$\|\Phi_H^t(y) - \Phi_H^t(z)\| \leq C_1\, d(y,z)$.
\end{lemma}
 
\begin{proof}
Assume for simplicity that $t>0$ the computation for a negative iterate $t$ is analogous. Since $\Phi^t_H(x)$ is a solution of the differential equation $\partial_t{u}(t)=H(X^t(x)) \cdot u(t)$ we obtain that 
$$\Phi_H^t(x)v=v+\int_0^t H(X^s(x)) \Phi_H^s(x)v \,ds,$$ for all $t$. Therefore, it follows from Gronwall's inequality (see e.g. \cite{PdM}) that $\|\Phi_H^t(x)v\| \le  e^{\|H\| |t|}\|v\|$ for all $t\in\mathbb R$ and $v\in\mathbb R^{2\ell}$. Moreover, since $H$ is Lipschitz there exists $K>0$ so that

\begin{align*}
\|\Phi_H^t(y)v - \Phi_H^t(z)v  \|
	& \leq  \int_0^t \|H(X^s(y))-H(X^s(z))\| \|\Phi_H^s(y) v \| + \|H\| \| \Phi_H^s(y) v- \Phi_H^s(z)v\|\, ds\\
	& \leq    e^{|t| \|H\| } \|v\| K  \int_0^t e^{-\tau s} d(y,z) \,ds + \int_0^t \|H\| \| \Phi_H^s(y) v - \Phi_H^s(z) v \|\, ds\\
	& \leq   e^{|t| \|H\| } \|v\| K \tau^{-1}  d(y,z) + \int_0^t \|H\| \| \Phi_H^s(y) v - \Phi_H^s(z) v \|\, ds
\end{align*}
Applying again Gronwall's lemma it follows that  $\|\Phi_H^t(y)v - \Phi_H^T(z)v  \|  \leq e^{2 |t| \|H\| } \|v\| K  \,d(y,z)$
and, consequently, we obtain that for all $y,z\in M$
$
\|\Phi_H^t(y) - \Phi_H^t(z)  \|  \leq e^{2 t \|H\| } K  \,d(y,z),
$
which proves the lemma.
\end{proof}

The next result relates the Lyapunov spectrum of cocycles with the induced one.

\begin{lemma}\label{l.Lyapunov.spectrum.section}
Given a measurable map $f: M \to M$, an $f$-invariant, ergodic probability measure 
$\mu$  and a roof function 
$\varrho \colon M \to [0, +\infty)$ with $\varrho \in L^1(\mu)$ let $X^t$ be the suspension semiflow
 and  $\hat \mu=(\mu \times \Leb) / \int \varrho \, d\mu$ be an invariant probability measure.
Given $H\in C^{r,\nu}(M,\mathfrak{sp}(2\ell,\mathbb K))$ the cocycle $(\Phi_H^t,\hat \mu)$ has a 
non-zero Lyapunov exponent if and only if the same property holds for the cocycle 
$(\Psi_H,\mu)$ over $f$ given by 
$\Psi_H(x)=\Phi_H^{\varrho(x)}(x)$.
\end{lemma}

\begin{proof}
Since $\mu$ is ergodic, then the $X^t$-invariant probability measure $\hat\mu$ is also ergodic and the largest
Lyapunov exponent $\la^+(\Phi_H^t,\hat\mu)$ for the time-continuous cocycle $\Phi_H^t $ is given by
$
\la^+(\Phi_H^t,\hat \mu)
    =\lim_{t\to +\infty} \frac1t \log \|\Phi_H^t(x)\|
$
for $\hat\mu$-a.e. $x$. 
Moreover, by construction of $\hat\mu$, it follows from the Birkhoff ergodic theorem and ergodicity of $\mu$ that
$
\lim_{n\to +\infty} \frac{\varrho^{(n)}(z_0)}n= \int \varrho \; d\mu
$ 
and 
$
\la^+(\Psi_H,\mu)
     =\lim_{k \to +\infty} \frac1k \log \|\Psi_H^k (z_0)\| 
$
for $\mu$-almost every $z_0\in M\times \{0\}$, where 
\begin{equation}\label{inducing}
\varrho^{(n)}(x)=\sum_{0 \le j \le n-1} \varrho(f^j(x))).
\end{equation}
In particular, since $\varrho(z)$ denotes the first Poincar\'e hitting 
time of a point $z \in \hat M$ to the global section $M \times \{0\}$ then for $\mu$-almost every $z\in \hat M$
one has that $z_0=X^{\varrho(z)}(z)$ satisfies 
\begin{align*}
\la^+(\Phi_H^t,\mu)
    & =\lim_{t\to +\infty} \frac1t \log \|\Phi_H^{t-\varrho(z)}(X^{\varrho(z)}(z)) \; \Phi_H^{\varrho(z)} (z)\| = \lim_{n\to +\infty} \frac1{\varrho^{(n)}(z_0)} \log \|\Phi_H^{\varrho^{(n)}(z_0)}(z_0)\|. \\
    & =  \Big( \lim_{n\to +\infty} \frac{n}{\varrho^{(n)}(z_0)} \Big)
         \Big( \lim_{n\to +\infty} \frac1n \log \|\Psi_H^n(z_0)\| \Big)= \Big( \int \varrho \; d\mu \Big)^{-1} \la^+(\Psi_H, \mu),
\end{align*}
where $\la^+(\Psi_H,\mu)$ denotes the maximum Lyapunov exponent of the
cocycle $\Psi_H$ with respect to $\mu$. In particular the largest Lyapunov exponent
of $\Phi_H^t$ with respect to $\mu$ is zero if and only the same holds for $\Psi_H$ with 
respect to $\mu$. Similar computations prove the same for the lowest Lyapunov exponent
$
\la^-(\Phi_H^t,\mu)
    =\lim_{t\to +\infty} \frac1t \log \|\Phi_H^t(x)^{-1}\|^{-1}
    \quad \text{ for $\hat\mu$-a.e. $x$. }
$
This finishes the proof of the lemma.
\end{proof}

\section{Hamiltonian linear differential systems over suspension flows}\label{sec:time.continuous.suspension}

In this section we prove Theorem~\ref{thm:flow} in the case of symplectic cocycles over 
nonuniformly hyperbolic suspension flows. Since the general case is more involving but uses
some of these arguments, this intermediate step from discrete-time to suspension flows will be useful 
to the reader.
As a consequence we will deduce Corollary~\ref{cor:hypflow} on 
symplectic cocycles over nonuniformly hyperbolic flows.  
The strategy to deal with Hamiltonian linear differential systems over  suspension flows is to make a reduction 
to the discrete-time setting, in which case we consider an induced cocycle in the fiber that also depends on the 
roof function. It is here that we need to require the roof function to be bounded (recall Lemma~\ref{l.Lyapunov.spectrum.section}). 
Moreover, an extra difficulty is caused since our perturbations are in the space of Hamiltonian linear differential system $H \in C^{r,\nu}(M, \mathfrak{sp}(2\ell,\mathbb K))$ as the infinitesimal generators 
of the fundamental solutions $\Phi_H^t$ over the
flow $(X^t)_t$. One of the main difficulties is really to analyze the variation of the holonomies for the cocycle
$\Psi_H$ as a function of the infinitesimal generators $H$.

So, given a compact Riemannian manifold $M_0$,  a $C^{1+\alpha}$ diffeomorphism $f: M_0 \to M_0$ endowed with an invariant and ergodic hyperbolic measure $\mu$ with local product structure and $\varrho \colon M_0 \to [0, +\infty)$  a H\"older continuous roof function which is bounded away from zero (without loss of generality assume the height is larger than  one), we consider the corresponding suspension flow $(X^t)_{t\ge 0}$ over $f$ acting on 
the space 
$
M = \{ (x,t) \in M_0\times \mathbb R_0^+ :  0 \le t \le \varrho(x)\} / \sim,
$
where $(x,\varrho(x))\sim (f(x),0)$, was described before in \S\ref{sec:suspension}.  It is clear from the definition that the $(X^t)_t$-invariant probability measure $\hat\mu=(\mu \times \Leb) / \int \varrho\, d\mu$ has the local product structure.

\subsection{A reduction to the base dynamics}\label{sec:reduction}

Our strategy to deal with suspension flows is to make a reduction of the dynamics, cocycle and invariant
measures by an inducing process.

\subsubsection{A cocycle reduction to the base dynamics}\label{sec:cocycle.reduction}

Our approach here is to use a reduction of the time-continuous Hamiltonian to the
case of discrete-time setting. For that purpose consider the cocycle 
$\Psi_H: \Sigma \times \mathbb K^{2\ell} \to \Sigma \times \mathbb K^{2\ell}$ induced from
$\Phi_H^t$ on the global cross-section $\Sigma=M_0\times\{0\}$ by
$
\Psi_H(x,v)=(f(x), \Phi_H^{\varrho(x)}(x) \, v).
$
Given $n \ge 1$ set 
$\Psi^n_H(x)=\Psi_H(f^{n-1}(x))\circ\dots\circ\Psi_H(f(x))\circ\Psi_H(x),$
and notice that 
$
\Psi^n_H(x)=\Phi^{\varrho^{(n)}(x)}_H(x),
$
where $\varrho^{(n)}$ is an inducing o
the original roof function and defined in (\ref{inducing}).
For simplicity reasons, we shall assume that the roof function $\varrho$ is Lipschitz and 
our first step, which is a generalization of Lemma~\ref{lem:timetau}, is to obtain the Lipschitz regularity for the 
induced cocycle. 

\begin{lemma}\label{Lips}
The induced cocycle $\Psi_H$ is Lipschitz continuous.
\end{lemma}

\begin{proof}
Since we assume that $H$ is Lipschitz, it follows from Lemma~\ref{lem:timetau} 
that the time-$t$ cocycle $\Phi_H^t$ is Lipschitz continuous for every $t\ge 0$. Hence, it follows that
$$
\|\Psi_H(x)-\Psi_H(y)\|
    \leq \|\Phi^{\varrho(x)}_H(x)-\Phi^{\varrho(x)}_H(y)\|
        + \|\Phi^{\varrho(x)}_H(y)-\Phi^{\varrho(y)}_H(y)\|.
$$
On the one hand, since $\varrho$ is continuous and $M_0$ is compact then $\varrho$ is bounded from above by the 
constant $\varrho_1$, the first term in the right hand side is bounded by $K( \varrho_1) d(x,y)$, for some $K(\varrho_1)>0$. 
On the other hand, since $\Phi_H^0=id$ and the roof function $\varrho$ is Lipschitz, then the 
rightmost term above is bounded by
$\|\Phi^{\varrho(x)}_H(y)\|\; \|\Phi_H^0(y)-\Phi^{\varrho(y)-\varrho(x)}_H(y)\| 
    \le	e^{K \varrho_1} C |\varrho(y)-\varrho(x)|
    \le C' d(x,y)$
for some positive constants $C, C'$. This proves the lemma.
\end{proof}

In the remaining of this section we assume that $H\in C^{r,\nu}(M,\mathfrak{sp}({2\ell,\mathbb K}))$ is such that all the Lyapunov exponents for $\Phi_H^t$ with respect to $\hat \mu$ are equal to zero.  It follows from 
Lemma~\ref{l.Lyapunov.spectrum.section} that the discrete-time cocycle $\Psi_H$ over $(f,\mu)$ has only zero
Lyapunov exponents. So, there exist constants $K, \tau, N, \theta$ 
such that $3\theta<\tau$ and the holonomy block $\cO=\mathcal H(K, \tau) \cap \mathscr{D}_{\Psi_H}(N,\theta)
\subset \Sigma$, for the cocycle $\Psi_H$, has positive $\mu$-measure.
In consequence we obtain from \cite[Proposition 4.2]{BeVar} the existence of unstable
holonomies. More precisely,

\begin{corollary}\label{holonomiaL}
For every $x\in \cO$ and $y\in W_{\text{loc}}^u(x)\subset \Sigma$, there exists $C_1>0$ and a symplectic linear
transformation $L^u_{H,x,y}:\{x\} \times P \mathbb K^{2\ell} \to \{y\} \times P \mathbb K^{2\ell}$ 
such that:
\begin{enumerate}
\item $L^u_{H,x,x}=id$ and $L^u_{H,x,z}=L^u_{H,y,z}\circ L^u_{H,x,y}$;
\item $\Psi_H(f^{-1}(y)) \circ L^u_{H,f^{-1}(x),f^{-1}(y)} \circ \Psi_H(x)^{-1}= L^u_{H,x,y}$;
\item $\|L^u_{H,x,y}-id\|\le C_1 d(x,y)$ and
\item $L^u_{H,f^j(y), f^j(z)}=\Psi_H^j(z) \circ L^u_{H,y,z}\circ  \Psi_H^j(y)^{-1}$ for all $j\in \mathbb Z$
\end{enumerate}
for every $x,y,z$ in the same local unstable manifold.
\end{corollary}

Notice the Poincar\'e return map $f$ on the global cross-section $\Sigma$ preserves $\mu$ with the local
product structure property. For any positive measure holonomy block $\cO$ and for all  $x \in \supp(\mu\mid \cO)$ 
let $\cN_{x}(\cO,\delta)$,   $\cN^u_{x}(\cO,\delta)$ and $\cN^s_{x}(\cO,\delta)$  be the induced neighborhoods 
of $x$ in $\mathcal O$ with  local product structure defined similarly as before.
Moreover, by some abuse of notation, when no confusion is possible on the Hamiltonian $H$ we shall denote 
by $h^s_{x,y}$ and $h^u_{x,y}$ the projectivization of the stable and unstable holonomies, respectively.

\subsubsection{Invariant measures reduction to the base dynamics}\label{sec:measure.reduction}

Now we relate invariant measures for the original and induced cocycles. 
Let $(\varphi_H^t)_t$ denote the cocycle over $(X^t)_t$ projectivized from $(\Phi_H^t)_t$
and we will also denote by $\psi_H$  the projectivized cocycle obtained from $\Psi_H$. 

\begin{lemma}\label{lem:reduction1}
Let $m$ be a $(\varphi_H^t)_t$-invariant probability measure such that $\Pi_*m=\hat\mu$ and
$m=\int_M m_z\; d\hat\mu(z)$ be a disintegration of $m$.  Then $(\varphi_H^t(z))_*m_z=m_{X^t(z)}$
for $\mu$-almost every $z\in\Sigma$ and all $0\le t\le \varrho(z)$.
\end{lemma}

\begin{proof}
Notice that, by construction,  $\{(x,\varrho(x)):x\in M_0\} \subset M$ is a zero $\hat\mu$-measure set and, consequently, 
$
\hat \mu (\{ (x,t) \in M_0\times \mathbb R_0^+ :  0 \le t < \varrho(x)\} )=1.
$
Moreover, the family $(\{z\} \times [0,\varrho(z)))_{z\in \Sigma}$ defines a measurable partition of the previous set in the
sense of Rokhlin. The same holds for the partition into points of each segment of orbit $\{z\} \times [0,\varrho(z))$. 
Therefore, given any $(\varphi_H^t)_t$-invariant probability measure  $m$ such that $\Pi_*m=\hat\mu$ there
exists a $\hat \mu$-almost everywhere defined family of probability measures $(m_{X^t(z)})_{z\in\Sigma, t\in[0,\varrho(z))}$ 
such that $\supp(m_{X^t(z)})\subset \{X^t(z)\}\times  P \mathbb K^{2\ell}$ and $m=\int m_{X^t(z)}\; d\hat\mu$.
More precisely, using $\hat\mu = (\mu\times \Leb)/\int \varrho\, d\mu$ one can write 
$$
m(E)= \frac{1}{\int \varrho\; d\mu}\int \left[ \int_0^{\varrho(z)} m_{X^t(z)}(E) \;dt \, \right] d\mu(z)
$$
for all measurable sets $E\subset M\times   P \mathbb K^{2\ell}$.
By invariance of $m$, for all $t\in\mathbb R$ we get that $(\varphi_H^t)_* m=m$ and, since any two disintegrations of the same probability measure coincide almost everywhere we get $m_{X^t(z)}=
(\varphi_H^t(z))_*{m_z}$ for $\mu$-almost every $z\in\Sigma$ and Lebesgue almost every $t\in[0,\varrho(z))$.
Finally, just observe that one can consider on each fiber $\{z\}\times[0,\varrho(z))$ for the disintegration 
the elements given by $t\mapsto (\varphi_H^t(z))_*{m_z}$, that vary continuously with $t$ in the weak$^*$ topology.
\end{proof}

\subsection{Continuous disintegration and criterion for non-zero Lyapunov exponents}\label{sec:criteria.suspension}

Here we just collect some of the previous ingredients and prove that all zero Lyapunov exponents for the 
time-continuous cocycle $\varphi_H^t$ over $(X^t,\hat\mu)$ implies on a rigid condition on the disintegration of
$\psi_H$-invariant measures. Throughout, let $m$ be a a $(\varphi_H^t)_t$-invariant probability measure such that $\Pi_*m=\hat\mu$. 

\begin{lemma}
The measure $m$ is completely determined by a probability measure $m_\Sigma$ on $\Sigma\times  P\mathbb K^{2\ell}$
such that $(\psi_H)_*m_\Sigma=m_\Sigma$ and $\Pi_*m_\Sigma=\mu$.
\end{lemma}

\begin{proof}
Lemma~\ref{lem:reduction1} implies that $m$ is completely determined by probability measures $(m_z)_{z\in \Sigma}$. 
Moreover, it is clear from the invariance condition that $m_\Sigma=\int m_z \; d\mu$ is $\psi_H$-invariant, because we have $(\varphi_H^t(z))_*m_z=m_{X^t(z)}$ for all $0\le t<\varrho(z)$ and continuity in the weak$^*$ topology.

On the one hand, since $\Pi_* m=\hat \mu$, for any measurable cylinder 
$E=E_1\times[0,b]$ contained in $\{ (z,t) : z\in M_0 \text{ and } 0\le t <\varrho(z) \}$ we get 
$$
(\Pi_* m)(E)=m(\pi^{-1}(E))=\hat \mu(E)  = \frac{1}{\int \varrho\; d\mu}\times \mu(E_1) \times \int_0^b dt.
$$

On the other hand, using the invariance condition $(\varphi_H^t(z))_*m_z=m_{X^t(z)}$ then one can write
$$
m=\int_M m_z \,d\hat \mu(z) 
	= \frac{1}{\int \varrho\; d\mu} \int_{M_0} \left[ \int_0^{\varrho(z)} m_{X^t(z)} \; dt \right]\; d\mu(z).
$$

Therefore for the set $E$ as above we get
$$
(\Pi_* m)(E)=m(\pi^{-1}(E))=\frac{1}{\int \varrho\; d\mu}\int_0^b m_\Sigma (\pi^{-1}(E)\cap\Sigma) dt= \frac{1}{\int \varrho\; d\mu}\Pi_*m_\Sigma(E) \times \int_0^b dt.
$$
By continuity, it follows that $(\varphi_H^{\varrho(z)})_*{m_z}=m_{f(z)}$ for $\mu$-almost every $z$ and consequently we have 
$(\psi_H(z))_*m_z=m_{f(z)}$ for $\mu$-almost every $z$. This shows that $(\psi_H)_*m_\Sigma=m_\Sigma$ and 
$\Pi_*m_\Sigma=\mu$, finishing the proof of the lemma.
\end{proof}

This put us in a position to make use of Proposition 4.3 in \cite{BeVar} applied to the cocycle $\psi_H$ that yields the following immediate consequence
(a similar result holds for unstable holonomies):
\begin{corollary}\label{measures}
Let $\cO$ be a positive $\mu$-measure holonomy block, consider $x \in \supp(\mu\mid \cO)$ and 
set the neighborhoods $\cN_{x}(\cO,\delta)$ of $x$ as above. Then, every $\psi_H$-invariant probability 
measure $m_\Sigma$ with $\Pi_*m_\Sigma=\mu$ admits a continuous disintegration on 
$\supp(\mu\mid \mathcal N_x(\cO,\de))$.
Moreover, 
\begin{equation}\label{exact}
m_{z}=(h^s_{H,y,z})_* m_{y}
	\quad \text{and} \quad 
	m_{z}=(h^u_{H,w,z})_* m_{w}
\end{equation}
for all $y,z,w \in \supp(\mu\mid \mathcal N_x(\cO,\de))$ such that $y,z$ belong to the same strong-stable local manifold
and   $z,w$ belong to the same strong-unstable local manifold.
\end{corollary}
Let $\cO$ be a positive $\mu$-measure holonomy block for a cocycle $\Psi_H$
and let $x \in \supp(\mu\mid \cO)$ be as above. 
By continuity of the disintegration and Lemma 4.4 in \cite{BeVar}, if $m_\Sigma$ is an $\psi_H$-invariant probability measure such that $\Pi_*m_\Sigma=\mu$ and  $p\in \supp(\mu\mid \mathcal N_x(\cO,\de))$ is a periodic point of period $\pi$ for $f$ then 
${\Psi_H^{\pi}(p)}_* m_p= m_p$ and, in addition, if $\Psi_H^{\pi}(p)$ has all real and distinct eigenvalues, 
then there exist elements  $\{v_i\}_{i=1\dots 2\ell}$ in $P \mathbb K^{2\ell}$ and a probability vector $\{\al_i\}_{i=1\dots 2\ell}$ such that  $m_p=\sum_{i=1}^{\ell} \al_i \, \delta_{v_i}$. 
Moreover, as in \cite[Corollary 4.5]{BeVar}, given $p,q\in \supp(\mu\mid \mathcal N_x(\cO,\de))$ dominated periodic points for $f$ and $z$ is the unique point in the heteroclinic intersection $W_{\text{loc}}^u(q) \cap 
W_{\text{loc}}^s (p)$ then
$
m_z=(h^u_{p,z})_*m_p=(h^s_{q,z})_*m_q.
$
Recall that if $p\in \Sigma$ is a dominated periodic point of period $\pi$ for $f$, then 
$h^s_{p,z}$ is the projectivization of the stable holonomy for the cocycle $\Psi_H$ over $f$, which is given by 
\begin{equation}\label{eq:L-H}
L^s_{H,p,z}=\lim_{n\to\infty} \Psi_H^{\pi n}(z)^{-1} \Psi_H^{\pi n}(p).
\end{equation}



\subsection{Realization of symplectic maps by Hamiltonian linear differential systems}\label{PH}

In the present section we show that for any given Hamiltonian linear differential system and any small perturbation of the symplectomorphism given by the time-one map of its solution, there exists a Hamiltonian linear differential system close to the original one which realizes the perturbation map (cf. Lemma~\ref{perturb2} below).

Given $S\in  sp (2\ell,\mathbb K)$, we can view $S$ as the time-one map of the linear flow solution of the linear variational equation $\dot{u}(t)=\textbf{S}(t)\cdot u(t)$ with initial condition $u(0)=id$. In other words, $u(t)=\Phi^t_\textbf{S}$ is solution of $\dot{u}(t)=\textbf{S}(t)\cdot u(t)$, and $\Phi^1_\textbf{S}=S$.
Since, by Gronwall's inequality we have
\begin{equation}
\|S_t\|_{r,\nu}
    \leq \exp\left\{\int_0^t \|\textbf{S}(s)\|_{r,\nu}  \,ds\right\},\forall t \geq 0
\end{equation}
we say that $S\in  sp (2\ell,\mathbb K)$ is $\delta$-$C^{r,\nu}$-close to identity if $\textbf{S}$ is $\delta$-$C^{r,\nu}$-small, i.e., $\|\textbf{S}\|_{r,\nu}<\delta$.

\begin{lemma}\label{perturb2}
Let be given $H\in C^{r,\nu}(M, \mathfrak{sp} (2\ell,\mathbb K))$ over a flow $X^t\colon M\rightarrow M$, any nonperiodic point $x\in M$ (or periodic with period $>1$) and  $\epsilon>0$. There exists $\delta=\delta(H,\epsilon)>0$ such that if $S\in  sp (2\ell,\mathbb K)$ is isotopic to the identity and $\delta$-$C^{r,\nu}$-close to $id$, then there exists $H_0\in  C^{r,\nu}(M, \mathfrak{sp} (2\ell,\mathbb K))$ satisfying:
\begin{enumerate}
\item [(a)] $\|H_0-H\|_{r,\nu}<\epsilon$ and
\item [(b)] $\Phi_{H_0}^1(x)=\Phi_H^1(x)\circ S$.
\end{enumerate}

\end{lemma}

\begin{proof}
By the tubular flowbox theorem (see ~\cite{AM,PdM}) there exists a smooth change of coordinates so that there exists a 
local conjugacy of $X$ on a neighborhood of the segment of orbit $\{X^{t}(x)\colon t\in[0,1]\}$ to a constant vector field on $\mathbb{R}^d$ where $d=2\ell=\dim(M)$. With this assumption we consider $x=\vec{0}$ and 
$\{X^{t}(x)\colon t\in[0,1]\}=\{(t,0,...,0)\in\mathbb{R}^d\colon t\in[0,1]\}\subset \frac{\partial}{\partial x_1},$ where $\frac{\partial}{\partial x_1}$ denotes the direction spanned by direction $x_1=(1,0,...,0)$.  Given $\rho>0$ let $B(\vec{0},\rho)\subset \left(\frac{\partial}{\partial 1}\right)^{\perp}$ denotes the ball centered in $\vec{0}$ of radius $\rho$ contained in the hyperplane orthogonal to $\frac{\partial}{\partial x_1}$. 
The perturbation will be performed in the cylinder $\mathcal{C}=B(\vec{0},\rho)\times [0,1]=\{X^{t}(B(\vec{0},\rho))\colon t\in[0,1]\}.$
Using the fact that $M$ is compact we can take
\begin{equation}\label{K}
K:=\underset{z\in M, t\in[0,1]}{\max}\{ \|\Phi_{H}^{t}(z)\|_{r,\nu} , \|(\Phi_{H}^{t}(z))^{-1}\|_{r,\nu},\|H\|_{r,\nu} \}.
\end{equation}
Fix any $\epsilon>0$ and choose 
$
\delta:=\frac{\epsilon}{ 6K^3}.
$
Consider an isotopy $S_t\in  sp (2\ell,\mathbb K)$,$t\in[0,1]$, such that:
\begin{enumerate}
\item $S_t=(1-t)id+t S$;
\item $S_t$ is the solution of the linear variational equation $\partial_t{u}(t)=\textbf{S}(t)\cdot u(t)$ with infinitesimal generator $\textbf{S}:=S_t^\prime=S-id$ satisfying the inequality 
$$\|\textbf{S}\|_{r,\nu}:=
	\underset{0\leq j \leq r}{\sup}  \|D^j \textbf{S}(t)\|
	+\underset{t\not=s}{\sup}\frac{\|\textbf{S}(t)-\textbf{S}(s)\|}{|t-s|^\nu}<\delta.$$
\end{enumerate}
Consider a $C^\infty$ \emph{bump-function} $\alpha\colon [0,\infty[\rightarrow [0,1]$, with $\alpha(s)=0$ if $s\geq \rho$ and $\alpha(s)=1$ if $s\in[0,\rho/2]$. Given $z\in B(\vec 0,\rho)$ consider the linear isotopy $S_t(z)\in  sp (2\ell,\mathbb K)$, $t\in[0,1]$, 
between $S_0(z)=id$ and $S_1(z)=\alpha(\|z\|^2)S$ obtained as solution of the equation
$\partial_t{u}(t,z)=\textbf{S}(t,z)\cdot u(t,z)$ with infinitesimal generator $\textbf{S}$ satisfying
$$
\|\textbf{S}(t,z)\|_{r,\nu}:=
	\underset{0\leq j \leq r}{\sup} \; \underset{t\in [0,1]}{\sup} \|D^j \textbf{S}(z+(t,0,...,0))\|
	+\underset{x\not=y}{\sup}\frac{\|\textbf{S}(x)-\textbf{S}(y)\|}{d(x,y)^\nu}<\delta.
$$
Then, if $\Upupsilon_t(z)=\Phi^t_H(z)\alpha(\|z\|)S_t(z)$ and we consider time derivatives one notices that
\begin{eqnarray*}
\Upupsilon_t(z)^{\prime}&=& \Phi_{H}^{t}(z)^{\prime}\alpha(\|z\|)S_t(z)+\Phi_{H}^{t}(z)(\alpha(\|z\|)S_t(z))^{\prime}\\
&=& H(X^{t}(z))\Phi_{H}^{t}(z)\alpha(\|z\|)S_t(z)+\Phi_{H}^{t}(z)(\alpha(\|z\|)S_t(z))^{\prime}\\
&=& H(X^{t}(z))\Upupsilon_{t}(z)+\Phi_{H}^{t}(z)(\alpha(\|z\|)S_t(z))^{\prime}(\Upupsilon_{t}(z))^{-1}\Upupsilon_{t}(z)\\
&=& \left[H(X^{t}(z))+P(X^{t}(z))\right]\cdot \Upupsilon_{t}(z)
\end{eqnarray*}
where $P(X^{t}(z))=\Phi_{H}^{t}(z) \textbf{S}(t,z)(\Phi_{H}^{t}(z))^{-1}$ in the flowbox coordinates 
$(z,t)\in\mathcal{C}=B(\vec 0,\rho) \times [0,1]$ and outside the flowbox cylinder $\mathcal{C}$ we let $P=[0]$. 
Consequently $\Upsilon_t$ is a solution of the equation $\partial_t{u}(t,z)=H_0(X^{t}(z) \cdot u(t,z)$ with
initial condition the identity, where $H_0(X^{t}(z))=H(X^{t}(z))+P(X^{t}(z))$ for all $t\in[0,1]$ and $z\in B(\vec0,\rho)$.
Notice also that
\begin{eqnarray*}
P(X^{t}(z))&=&\Phi_{H}^{t}(z)(\alpha(\|z\|)S_t(z))^{\prime}(\Upupsilon_{t}(z))^{-1}\\
&=&\Phi_{H}^{t}(z)\alpha(\|z\|)S_t(z))^{\prime}(\Phi^t_H(z)\alpha(\|z\|)S_t(z))^{-1}\\
&=&\Phi_{H}^{t}(z)\alpha(\|z\|) S^{\prime}_t(z) (S_t(z))^{-1}\alpha(\|z\|)^{-1} (\Phi_{H}^{t}(z))^{-1}\\
&=&\Phi_{H}^{t}(z)S^{\prime}_t(z) (S_t(z))^{-1} (\Phi_{H}^{t}(z))^{-1}.
\end{eqnarray*}
Using this, we claim that $\textbf{S}:=(S_{t})^{\prime}(S_{t})^{-1}\in\mathfrak{sp}(2\ell,\mathbb{R})$ for all $t$ 
and $z$ and, consequently, $P(X^{t}(z))\in\mathfrak{sp}(2\ell,\mathbb{R})$. It is enough to prove that 
$J\textbf{S}+\textbf{S}^TJ=0$. Since for any $S\in sp(2\ell,\mathbb{R})$ we have the symplectic identities $J^{-1}=J^T=-J$, $S^TJS=J$ and $S^{-1}=J^{-1}S^TJ$ then 
\begin{eqnarray*}
J\textbf{S}+\textbf{S}^TJ&=& J(S_{t})^{\prime}(S_{t})^{-1}+[(S_{t})^{\prime}(S_{t})^{-1}]^TJ\\
&=& J(S_{t})^{\prime}(S_{t})^{-1}+((S_{t})^{-1})^T((S_{t})^{\prime})^TJ\\
&=& J[-(S_{t})^{\prime}(S_{t})^{-1}J-J((S_{t})^{-1})^T((S_{t})^{\prime})^T]J\\
&=& J[(S_{t})^{\prime}J^{-1}(S_{t})^T+S_t J^{-1}(S_{t}^{\prime})^T]J\\
&=& J[S_{t}J^{-1}(S_{t})^T]^{\prime}J\\
&=& J[J^{-1}]^{\prime}J=0
\end{eqnarray*}
which proves our claim.

Second, we will prove condition (a) of the conclusions of the lemma, i.e., that $\|H_0-H\|_{r,\nu}<\epsilon$ or, equivalently, that $\|P\|_{r,\nu}<\epsilon$. We will perform the computations for $r=0$ with all the details. For $r\in\mathbb{N}$ we can estimate easily using the chain rule and Cauchy-Schwarz inequality. Whenever we consider points $x,y$ in the tubular flowbox $\mathcal C$ (the support of the perturbation) we write them in the flowbox coordinates $x=(z,t)$, $y=(w,s)$, where $t,s\in[0,1]$ and $z,w\in B(\vec 0,\rho)$. 

We shall estimate $P$ in both coordinates and then the estimates on $\|P\|_{0,\nu}=\|P\|_{\nu}$ can be obtained on $B(\vec0,\rho)\times[0,1]$ by means of a triangular inequality argument.

If $z_t,w_t$ inside the same laminar section in $\mathcal C$, i.e., $z_t=(z,t)$ and $w_t=(w,t)$ then
using (\ref{K}) it follows that
\begin{align*}
\|P(z_t)-P(w_t)\| 
	& =\| \Phi_{H}^{t}(z)  \textbf{S}(t,z) (\Phi_{H}^{t}(z))^{-1}- \Phi_{H}^{t}(w)  \textbf{S}(t,w) (\Phi_{H}^{t}(w))^{-1}\| \\
	& \leq \| \Phi_{H}^{t}(z)  [\textbf{S}(t,z)-\textbf{S}(t,w)] (\Phi_{H}^{t}(z))^{-1}\|
	+ \| [ \Phi_{H}^{t}(z) -\Phi_{H}^{t}(w)]  \textbf{S}(t,w) (\Phi_{H}^{t}(z))^{-1}\| \\
	& + \| \Phi_{H}^{t}(w)  \textbf{S}(t,w) [(\Phi_{H}^{t}(z))^{-1}-(\Phi_{H}^{t}(w))^{-1}]\| \\
	& \leq K^2\| \textbf{S}(t,z)-\textbf{S}(t,w)\| + K\|  \Phi_{H}^{t}(z) -\Phi_{H}^{t}(w)\|\|  \textbf{S}(t,w) \| \\
	& + K \| \textbf{S}(t,w) \| \, \| (\Phi_{H}^{t}(z))^{-1}-(\Phi_{H}^{t}(w))^{-1} \| \\
\end{align*}
and so
\begin{eqnarray*}
\underset{z_t\not=w_t}{\sup}\frac{\|P(z_t)-P(w_t)\|}{d(z_t,w_t)^\nu}
		\leq K^2\delta+2Ke^\delta \delta< \epsilon.
\end{eqnarray*}

Analogously, for $z_t,z_s$ inside the same orbit in $\mathcal C$ it follows

\begin{eqnarray*}
\|P(z_t)-P(z_s)\|
	& = & \| \Phi_{H}^{t}(z)  \textbf{S}(t,z) (\Phi_{H}^{t}(z))^{-1}- \Phi_{H}^{s}(z)  \textbf{S}(s,z) (\Phi_{H}^{s}(z))^{-1}\| \\
	&\leq & \| \Phi_{H}^{t}(z)  [\textbf{S}(t,z)-\textbf{S}(s,z)] (\Phi_{H}^{t}(z))^{-1}\|
		+ \| [ \Phi_{H}^{t}(z) -\Phi_{H}^{s}(z)]  \textbf{S}(s,z) (\Phi_{H}^{t}(z))^{-1}\| \\
	& + &\| \Phi_{H}^{s}(z)  \textbf{S}(s,z) [(\Phi_{H}^{t}(z))^{-1}-(\Phi_{H}^{s}(z))^{-1}]\| \\
	&\leq & K^2\| \textbf{S}(t,z)-\textbf{S}(s,z)\| +K\|  \Phi_{H}^{t}(z)(id -\Phi_{H}^{s-t}[X^t(z)])\| \|  \textbf{S}(t,w) \| \\
	&+ & K\| \textbf{S}(t,w)\|\| (\Phi_{H}^{t}(z))^{-1}(id-(\Phi_{H}^{s-t}(X^t(z)))^{-1})\|
\end{eqnarray*}
and so
\begin{eqnarray*}
\underset{z_t\not=z_s}{\sup}\frac{\|P(z_t)-P(z_s)\|}{d(z_t,z_s)^\nu}
	& \leq & \underset{t\not=s}{\sup}\,K^2\delta+K^2\delta\frac{\|  id -\Phi_{H}^{s-t}(X^t(z))\|}{|t-s|^\nu}
		+K^2\delta\frac{\|  id-(\Phi_{H}^{s-t}(X^t(z)))^{-1}\|}{|t-s|^\nu}\\
	&\leq& K^2\delta+\underset{t\not=s}{\sup}\,K^2\delta\left(\frac{\|  id -\Phi_{H}^{s-t}(X^t(z))\|}{|t-s|}+\frac{\|  id-(\Phi_{H}^{s-t}(X^t(z)))^{-1}\|}{|t-s|}\right)\\
	& \le & K^2\delta+\underset{t\not=s}{\sup} \, K^2\delta \left(2\|H\|\right) \\
	& \leq & K^2\delta+2K^3\delta 
	\leq 3K^3\delta<\epsilon.
\end{eqnarray*} 

Notice that we consider $\nu=1$ (cf. the first paragraph on Section \S\ref{s.Lyap}). This is enough to deduce condition (a) using a triangular inequality argument.

Finally, we will prove condition (b) of the conclusions of the lemma, i.e., that we have the equality $\Phi_{H_0}^1(x)=\Phi_H^1(x)\circ S$. We are considering $x=\vec 0$ so let us prove that $\Phi_{H_0}^1(\vec0)=\Phi_H^1(\vec0) S$. Just observe that $\Upupsilon_t(z)$ is a solution of the linear differential equation 
\begin{equation}\label{H+P}
{u}^\prime(t,z)=[H(X^t(z))+P(X^t(z))]\cdot u(t,z)=H_0(X^t(z))\cdot u(t,z).
\end{equation}
But, given the initial condition $u(0,z)=z$, this solution is unique, say $\Phi^t_{H_0}(\vec0)$. Since $\Upupsilon_t(z)=\Phi^t_H(z)\alpha(\|z\|)S_t(z)$ and it also satisfies (\ref{H+P}) we obtain that $\Phi^t_{H_0}(\vec0)=\Phi^t_H(z)\alpha(\|z\|)S_t(z)$. Thus, for $z=\vec 0$ we get
$
\Phi^1_{H_0}(\vec0)
	=\Phi^1_H(\vec0)\alpha(\|z\|)S_1(\vec0)=\Phi^1_H(\vec0)\alpha(0)  S_1(\vec0)=\Phi^1_H(\vec0) S,
$
and the lemma is proved.
\end{proof}


\subsection{Proof of Theorem~\ref{thm:flow} for suspension flows} \label{sec:perturb.suspension}

This subsection is devoted to the proof of Theorem~\ref{thm:flow} on the existence of non-zero Lyapunov 
exponents for Hamiltonian linear differential systems over suspension flows. The first part is constituted by
some important perturbation arguments. 
Let  $(X^t)_t$ be a suspension flow of a $C^{1+\alpha}$ diffeomorphism $f:\Sigma \to \Sigma$ on a Riemannian manifold $\Sigma=M_0$ with Lipschitz continuous roof function $\varrho$. Recall that the case of H\"older roof function can be dealt by changing the metric. Assume that $(f,\mu)$ has local product structure and consider the
$(X^t)_t$-invariant probability measure $\hat\mu=(\mu \times \Leb) / \int \varrho \, d\mu$ with local product structure
with respect to the flow.

Let $H\in C^{r,\nu}(M,\mathfrak{sp}(2\ell,\mathbb K))$ be a Hamiltonian linear 
differential equation over the suspension flow $(X^t)_t$ such that  $\la^+(H,\mu)=0$, that is, so that $\Phi^t_H$ (hence also $\Psi_H$)  has only zero Lyapunov exponents at $\hat\mu$-almost everywhere. Take an arbitrary $\epsilon>0$ and also $k\ge 2$.  It is also a consequence of Proposition 4.7 in ~\cite{BeVar}, the perturbation Lemma~\ref{perturb2} and the boundeness of $\varrho$ that there exists a holonomy block $\mathcal O\subset \Sigma$ and exist distinct dominated periodic points $\{p_i\}_{i=1}^k$ by $f$ in the set $\mathcal O$ and a Hamiltonian linear differential system $\tilde H_0\in C^{r,\nu}(M, \mathfrak{sp}(2\ell,\mathbb K))$ satisfying
$\|H-\tilde H_0\|_{r,\nu}<\vep/2$ and such that $W_{\text{loc}}^{u}(p_i)\pitchfork W_{\text{loc}}^{s}(p_{i+1})\not= \emptyset$ consists of one point $z_i$ and the Lyapunov spectrum of 
$
\Psi^{\pi_i}_{\tilde{H}_0}(p_i)
	=\Psi_{\tilde{H}_0}(f^{\pi_i-1}(p_i))\dots\circ\Psi_{\tilde{H}_0}(f(p_i))\circ\Psi_{\tilde{H}_0}(p_i)
$
is real and simple, where $\pi_i$ 
is the period of 	the periodic point $p_i$ for all $i=1,...,k$. We point out that these perturbations are performed in the linear differential system and intend to realize the same scheme constructed in the discrete case.

\medskip

In order to go on with the proof we need to ``break the holonomy" by means of a small perturbation supported in the $k$ heteroclinic intersections. In the vein of the Breaking lemma (see \cite[Lemma 4.8]{BeVar}), we will show what to do in each one of the intersections. The following details should be taken into account: on one hand the holonomy properties described in Corollary~\ref{holonomiaL}, and to be used in the sequel, are with respect to $\Psi_H$ which is a well-behaved cocycle (cf. Lemma~\ref{Lips}), but on the other hand the perturbation is on the linear differential system $H$ and not in the (discrete) cocycle $\Psi_H$ similarly of what we did above. We let $W=\{w_i: i=1 \dots 2\ell\}$ be any linearly independent set of vectors in the fiber $P\mathbb{K}^{2\ell}$ over $z\in W^s_{loc}(p) \cap W^u_{loc}(q)$ where $p,q\in\Sigma$ are periodic $f$ orbits. Given $H\in C^{r,\nu}(M, \mathfrak{sp}(2\ell,\mathbb K))$ and a symplectic base  $\{v_i: i=1 \dots 2\ell\}$ in the fiber $P\mathbb{K}^{2\ell}$ over $p$, there exists 
a Hamiltonian linear differential system $H_0\in C^{r,\nu}(M, \mathfrak{sp}(2\ell,\mathbb K))$ such that $\|H-H_0\|_{r,\nu}<\vep$, the unstable holonomies 
coincide $L^u_{H_0,q,z}=L^u_{H,q,z}$  and $L^s_{H_0,p,z} (v_i)$ does not belong to the 1-dimensional subspace generated by $w_j$ for all $j$. Moreover, the later property is open in the $C^{r,\nu}$-topology. Let us see succinctly how to define $H_0$. 

First, we take a thin flowbox $U^\tau_{f^{\pi_1}(z)}$, where, $z\in\Sigma$, $U$ is a neighborhood of $f^{\pi_1}(z)$ in $\Sigma$ and $\pi_1$ the $f$-period of $p$. 
Second, using Lemma~\ref{perturb2} we perform a perturbation supported in $U^\tau_{f^{\pi_1}(z)}$ to spoil the strong accuracy of sending eigendirections into eigendirections by the holonomy as described in (\ref{exact}) of Corollary~\ref{measures} which we remind in (\ref{exact2}). Recall that $\psi_H$ is the projectivized cocycle obtained from $\Psi_H$. By Corollary~\ref{measures} we get that, under our context, every $\psi_H$-invariant probability 
measure $m_\Sigma$ with $\Pi_*m_\Sigma=\mu$ admits a continuous disintegration.
Moreover, denoting by $h^s_{H,p,z}$  the projectivization of the stable holonomy for the cocycle $\Psi_H$ over $f$ and by $h^u_{H,q,z}$ the projectivization of the unstable holonomy for the cocycle $\Psi_H$ over $f$ we get
\begin{equation}\label{exact2}
m_{z}=(h^s_{H,p,z})_* m_{p}
	\quad \text{and} \quad 
	m_{z}=(h^u_{H,q,z})_* m_{q}
\end{equation} 
for some points $q,z,p$ such that $p,z$ belong to the same strong-stable local manifold
and   $z,q$ belong to the same strong-unstable local manifold.

Now, as in the discrete case we intend to deal with $k$ periodic orbits $p_i$ and perform $k$ disjoint supported perturbations $H_i$ which are different form $H$ exactly in $U^\tau_{f^{\pi_i}(z_i)}$, where $\pi_i$ is the period of $p_i$ and $z_i$ its associated heteroclinic orbit. We observe that once we let $H_i=H$ outside $U^\tau_{f^{\pi_i}(z_i)}$ and noting that $z_i$ is homoclinic with $p_i$ and $q_i$ there is no way for us to interfere on the unstable holonomy, thus $L^u_{H_i,q_i,z_i}=L^u_{H_i,q_i,z_i}$ for all $i=1,...,k$. One also has that $L^s_{H_i,p_i,f^{2\pi_i}(z_i)}=L^s_{H,p_i,f^{2\pi_i}(z_i)}$. Actually, we can be more precise because, since the perturbation is performed in $U^\tau_{f^{\pi_i}(z_i)}$, we have $L^s_{H_i,p_i,f^{\pi_i+1}(z_i)}=L^s_{H,p_i,f^{\pi_i+1}(z_i)}$. Now, using (\ref{eq:L-H}) we obtain
\begin{align}\label{eq.interates2}
L^s_{H,p_i,z_i}=\lim_{n\to\infty} \Psi_H^{\pi_i n}(z_i)^{-1} \Psi_H^{\pi_i n}(p_i)
	= [\Psi^{2  \pi_i} (z_i)]^{-1} \; L^s_{H,p_i,f^{2 \pi_i}(z_i)}.
\end{align}
Notice that the sets $U_{f^{\pi_i}(z_i)}^\tau$ are pairwise disjoint, so we define $H_0$ to be equal to $H_i$ inside  $U_{f^{\pi_i}(z_i)}^\tau$ and $H_i=H$ in $M\setminus \cup_{i=1}^{k} U_{f^{\pi_i}(z_i)}^\tau$. Furthermore, each Hamiltonian linear differential system $H_i$ should be such that
$
L^s_{H_i,p_i,z_i}(v_i) 
	= [\Psi_{H_i}^{2  \pi_i} (z_i)]^{-1} \; L^s_{H_i,p_i,f^{2 \pi_i}(z_i)} (v_i)
	= [\Psi_{H_i}^{2  \pi_i} (z_i)]^{-1} \,  (e_i),
$
does not belong to any subspace generated by proper subsets of $W$. 
Finally, the perturbation $H_0$ will be such that for all $i$ we have
$$
(h^u_{H_0,p_i,z_i})_*m_{p_i} \neq (h^s_{H_0,p_{i+1},z_i})_*m_{p_{i+1}}.
$$
Together with Corollary~\ref{measures} this implies that $\Psi_{H_0}$ has at least one non-zero Lyapunov exponent
and proves that the set of linear differential systems in  $C^{r,\nu}(M, \mathfrak{sp}(2\ell,\mathbb K))$ over suspension flows with bounded roof function and with at least one non-zero Lyapunov exponent is an open and dense set.

\section{Hamiltonian linear differential systems: General case}\label{sec:time.continuous.general}

In this section we prove Theorem~\ref{thm:flow} in the case of Hamiltonian skew-product flows over 
general nonuniformly hyperbolic flows. Some of the main differences with the case of suspension flows
is that typically strong stable and unstable foliations are not jointly integrable and, consequently, one cannot
a priori build global cross sections and apply directly the results concerning holonomy invariance for discrete
time maps.  In fact, not only the construction of good return time map is also more involving as one needs to
prove that good hypebolicity properties are inherited by projection of the local dynamics to the local cross section. 

The strategy here is to prove that nonuniform hyperbolicity for the flow yields
some nonuniform hyperbolicity of the Poincar\'e first return map to some local smooth cross-section (recall Lemma~\ref{Pesin}). Then, 
one induced a discrete-time cocycle and reproduce the ideas from the suspension flow setting back in Section \ref{sec:time.continuous.suspension}.  
The key arguments are to define properly an invertible return map with nonuniform hyperbolicity
and to keep track of the local product structure. 

As before we endow $C^{r,\nu}(M, \mathfrak{sp} (2\ell,\mathbb K))$ with the
$C^{r,\nu}$-topology defined using the norm 
$$
\|H\|_{r,\nu}
	=\underset{0\leq j \leq r}{\sup} \; \underset{x\in M}{\sup} \|D^j H(x)\|
	+\underset{x\not=y}{\sup}\frac{\|H(x)-H(y)\|}{\|x-y\|^\nu},$$
where $H\in C^{r,\nu}(M, \mathfrak{sp} (2\ell,\mathbb K))$ and $x,y\in M$.

\subsection{Non-uniform hyperbolicity for the flow and hyperbolicity for local Poincar\'e maps}\label{nh}

Let us consider a smooth flow $X^t\colon M\rightarrow M$ preserving a hyperbolic and ergodic probability measure
$\hat \mu$ with local product structure as in Definition~\ref{def:lpsf}.
By Pesin theory,  there exists a $\hat\mu$-full measure set $\mathscr{P}$ and measurable functions $K\colon \mathscr{P}\rightarrow (0,+\infty)$ and
$\tau \colon \mathscr{P}\rightarrow (0,+\infty)$ so that for any given $x\in \mathscr{P}$ there is a well defined local stable manifold $W^{s}_{loc}(x)$ such that $T_x W^{s}_{loc}(x)=E^s_x$ and $d(X^t(y),X^t(z))\leq K_x \,e^{-\tau_x\,t}d(y,z)$, for every $y,z\in W^{s}_{loc}(x)$ and $t\ge 0$. Similar property holds for local unstable manifolds. 

Therefore, given positive constants $K,\tau$ the points in the
hyperbolic block $\cH(K,\tau)$ are such that both the local invariant manifolds  $W_{\text{loc}}^s(x)$ and $W_{\text{loc}}^u(x)$ have uniform size, uniform contraction on $W_{\text{loc}}^s(x)$, uniform backward contraction
on $W_{\text{loc}}^u(x)$ and vary continuously with $x\in \cH(K,\tau)$.
In particular, if $X$ denotes the vector field  associated to $(X^t)_t$, i.e. $X(x):=\partial_t X^t(x)|_{t=0}$, then the angle between the vector $X(x)$ and any of the subspaces $E^s_x$ or $E^u_x$ varies continuously in $\cH(K,\tau)$ and consequently is bounded away
from zero on the hyperbolic block.
We proceed to build some projective hyperbolicity on some transversal cross-section to the flow.

\begin{lemma}\label{lem:returns}
Let $\Lambda$ be a positive $\hat\mu$-measure subset for the flow $(X^t)_t$. Given a regular point $x\in\supp (\hat\mu\mid_{\Lambda} )$ there exists a smooth local cross section $\Sigma$ to the flow at $x$ and a tubular neighborhood $U_x^\delta$ of $x$ such that $\hat\mu$-almost every $y\in U_x^\delta$
has infinitely many returns to $\Lambda\cap U_x^\delta$.
\end{lemma}

\begin{proof}
Let $x\in\supp (\hat\mu\mid_{\Lambda})$ for some set $\Lambda$ such that $\hat\mu(\Lambda)>0$ and $\hat\mu$-invariant with respect to $X^t$, then for any open set $V_x$ containing $x$ we have $\hat\mu(V_x)>0$. 
Since $x$ is regular we can consider, in local charts, a small cross section $\Sigma$ to $X(x)$, say normal to $X(x)$. Moreover, we have that the tubular neighborhood $U_x^\delta$ of $x$ defined by 
$U_x^\delta:=\{X^{t}(D_x)\colon t\in(-\delta,\delta)\}$, where $\delta>0$ and $D_x$ is a ball in the normal section to $X(x)$ centered in $x$, is such that $\hat\mu(U_x^\delta)>0$.
Observe that, in a neighborhood of $x$, we can decompose the measure $\hat\mu$ into $\hat\mu=\mu \times \Leb$ where $\mu_\Sigma$ denote the measure induced by $\hat\mu$ on $\Sigma$ by the projection $\pi$ along the flow direction determined by the tubular flow neighborhood and $\Leb$ is the length. Now, Poincar\'e recurrence assures that $\hat\mu$-a.e. $y\in U_x^\delta$, or equivalently $\mu$-a.e. $z\in D_x$, has infinitely many returns.
\end{proof}

Taking into account the previous result we can define by means of the tubular neighborhood theorem a time
$T>0$ and a smooth function $t:\Sigma_0 \to(-\delta,\delta)$ such that the Poincar\'e first return map
$
f_\Sigma : \Sigma_0 \to \Sigma 
$
is well defined by $f_\Sigma(y)=X^{T+t(y)}(y)$ in some open neighborhood $\Sigma_0\subset \Sigma$ of $x$ in the section. 
If $\Lambda\subset \cH(K,\tau)$ is a subset of a hyperbolic block, then one can define the foliations $\cF^s$ and $\cF^u$ for all $y,z\in \cH(K,\tau) \cap U_x^\delta$ and the intersection $[y,z]_{\Sigma_x}:=\cF^{u}_y \pitchfork \cF^{s}_{z}$ consists of a unique point, provided that $\delta$ is small. Moreover, by construction the foliations are invariant by the Poincar\'e map $f_\Sigma$.

Let $\mu_\Sigma$ be as in Lemma~\ref{lem:returns}. The measure $\mu_\Sigma$ is clearly invariant by $f_\Sigma$.  Then, by construction, not only $x\in \supp(\mu_\Sigma\mid_{\Lambda})$ as by the local product structure we
have that $\mu_\Sigma$ is equivalent to the product measure $\mu^u_x \times \mu^s_x$, where $\mu^i_x$ are 
the conditional measures on $\cF^i_x$. Moreover, 

\begin{proposition}\label{prop:local.construction}
Consider a positive $\hat\mu$-measure set $\Lambda\subset \cH(K,\tau)$. Given a regular point $x\in\supp (\hat\mu\mid_{\Lambda} )$ there exists a smooth local cross section $\Sigma$ and positive constants 
$K',\tau'$, such that for $\mu_\Sigma$-almost every $z\in \Sigma$ we have 
$
d(f^n_\Sigma(y),f^n_\Sigma(z))\leq K' \,e^{-\tau'\, n} d(y,z)
$ 
for all $n\ge 1$ and every $y,z$ in $\cF_z^s$. Similar statement holds for $\cF^u_z$ with respect to $f_\Sigma^{-1}$.
\end{proposition}

\begin{proof}
This is a direct consequence of Lemma~\ref{Pesin} and the subsequent paragraph.
\end{proof}

\subsection{Fiber-bunching for cocycles over Poincar\'e return maps}\label{ic}

Until the remaining of this section let $f_\Sigma$ be a Poincar\'e return map that has a ``good" hyperbolic
structure on the foliations $\cF^s$ and $\cF^u$. 
We consider the discrete-time induced cocycle over $(f_\Sigma,\mu_\Sigma)$ for $x\in \Sigma_0$ as
\begin{equation}\label{ic2}
\Psi_H(x) =\Phi_H^{T+t(x)}(x).
\end{equation}
It is not hard to see that $\Psi_H$ is Lipschitz continuous and that $(\Phi_H^t,\hat \mu)$ has only zero Lyapunov exponents if and if only if the same property holds also for the cocycle $(\Psi_H,\mu_\Sigma)$. In fact,

\begin{lemma}\label{2.4}
Assume that $\lambda^+(H,\mu)=0$, that is, $\Phi^t_H$ has only zero Lyapunov exponents.  Then for every 
$\vep>0$ there exists $T, \theta$ such that $\mu(\mathscr{D}_H(T,\theta))>1-\vep$.
\end{lemma}

\begin{proof}
Since it is analogous to Corollary~2.4 in \cite{Viana} we leave the details to the reader.
\end{proof}

We proceed to build stable and unstable holonomies for points in the same leaf of the foliations. 
Given $H\in C^{r,\nu}(M, \mathfrak{sp} (2\ell,\mathbb K))$, $N> 0$ and $\theta>0$, 
consider the set $\mathscr{D}_H(N,\theta)$ of points $x\in M$ satisfying 
\begin{equation}\label{ic3}
\prod_{j=0}^{k-1} \big\|  \Psi_H(f_\Sigma^{jN}(x))  \big\|  \; \big\|  \Psi_H(f_\Sigma^{jN}(x))^{-1}  \big\| \le e^{k N \theta}
	\quad\text{for all $k \ge 1$}
\end{equation}
and the dual relation for $\Phi_H^{-T}$ with relation to $X^{-T}$. 
Let $\mathcal O$ be a \emph{holonomy block for $H$} for $f_\Sigma$ if it is a compact subset of 
$\mathcal H(K, \tau) \cap \mathscr{D}_H(N,\theta)$ for some constants $K, \tau, T, \theta$ satisfying 
$3\theta<\tau$. 
Observe that domination is an open condition for the cocycle and this enables to obtain strong-stable and 
strong-unstable foliations for all nearby cocycles. More precisely, using that the foliations $\cF^s$ and $\cF^u$
inherit the hyperbolicity from the hyperbolic block (possibly with some larger constants) we prove the existence of holonomies similarly to the discrete-time setting. In fact, the same ideas as in \cite[Proposition 4.2]{BeVar} yields that:

\begin{proposition}\label{holono}
For every $x\in \cO$ and $y,z\in \cF_x^u$, there exists $C_2>0$ and a symplectic linear
transformation $L^u_{H,x,y}:\{y\} \times P \mathbb K^{2\ell} \to \{z\} \times P \mathbb K^{2\ell}$ 
such that:
\begin{enumerate}
\item $L^u_{x,x}=id$ and $L^u_{x,z}=L^u_{y,z}\circ L^u_{x,y}$
\item $\Psi_H(f^{-1}_\Sigma(z)) \circ L^u_{H,f^{-1}_\Sigma(y),f^{-1}_\Sigma(z)} \circ \Phi^t_H(y)^{-1}= L^u_{H,y,z}$ 
for all $t\ge 0$ and
\item $\|L^u_{H,y,z}-id\|\le C_2\, d(y,z)$.
\end{enumerate}
\end{proposition}

As before, some consequences are the continuous disintegration of invariant measures.  Given a  holonomy block $\cO$ and a regular point $x$ let $\cN^s_x(\cO,\delta)\subset \cN^s_x(\delta)$ 
the subset of $\Sigma_x$ obtained by replacing $\cH(K,\tau)$ by the holonomy block $\cO$ and define  
$\cN_x^u(\cO,\de)$ and $\cN_x(\cO,\de)$ analogously.
Let $\psi_H$ denote the projectivized version of the cocycle $\Psi_H$ and, 
by some abuse of notation, consider $h^u_{x,y}$ the projectivized holonomy $L^u_{H,x,y}$.
The next proposition asserts that one can obtain a continuous disintegration of invariant probabilities that
project to $\mu_\Sigma$. 

\begin{proposition}\label{prop:disinteg}
Let $\cO$ be a positive $\mu$-measure holonomy block, consider $x \in \supp(\mu_\Sigma\mid \cO)$ and 
set the neighborhoods $\cN^s_{x}(\cO,\delta), \cN^u_{x}(\cO,\delta)$ and $\cN_{x}(\cO,\delta)$ as above. Then every 
$\psi_H^t$-invariant probability measure $m$ with $\Pi_*m=\mu_\Sigma$ admits a continuous 
disintegration on $\supp(\mu_\Sigma \mid \mathcal N_x(\cO,\de))$. Moreover, 
$$
m_{z}=(h^s_{y,z})_* m_{y}
	\quad \text{and} \quad 
	m_{z}=(h^u_{w,z})_* m_{y}
$$ 
for all $y,z,w \in \supp(\mu\mid \mathcal N_x(\cO,\de))$ such that $y,z$ belong to the same $\cF^s$ leaf
and  $z,w$ belong to the same $\cF^u$ leaf.
\end{proposition}

\begin{proof}
This proof is analogous to \cite[Proposition 4.3]{BeVar} and uses the nonuniform hyperbolicity
of the Poincar\'e first return map in the same way as in \cite{Viana}. For that reason we shall omit the details 
and leave the proof as a thorough exercise to the reader, nonetheless is accomplished by borrowing the arguments in \cite[Proposition 3.1]{Viana}.
\end{proof}

\subsection{A lot of closed orbits inside holonomy blocks}\label{lot}

In order to go on with the proof of our results we also need to use the continuous-time version of \cite[Proposition 4.5]{Viana}, included in (1) and (2) of Proposition~\ref{KatokViana} below, and which is supported in Katok's shadowing lemma for nonuniformly hyperbolic systems proved in ~\cite{Kat}. We observe that the flows version of the Katok theorem was treated recently in a more general context   by Lian and Young (see \cite[\S1.2]{LY}). But before we state it we shall introduce some elementary notation typical of Pesin's theory framework. Recall that the hyperbolicity (uniform and nonuniform) for flows is often defined with respect to the linear Poincar\'e flow (cf. ~\cite{BR}). From now on we consider a $C^{1+\alpha}$ flow $X^t\colon M\rightarrow M$ on a $d$-dimensional manifold $M$.  Given $k=0,...,d-1$, $\ell>1$ and $\chi>0$ we denote by $\Lambda^k_{\chi,\ell}$ the set of points $x\in M$ defining a \emph{Pesin hyperbolic block} cf. ~\cite[\S2]{Kat} but with respect to the decomposition of the normal bundle at $x$, $N_x=N^s_x\oplus N^u_x$, where $\dim(N^s_x)=k$. The set $\Lambda_j^k:=\Lambda^k_{\chi_j,\ell_j}$ is the hyperbolic block with index $k$, Lyapunov exponent $\chi_j$, and constant $\ell_j$, where $\mu(\cup_i\Lambda_j^i)$ tends to $1$, and $\chi_j$ and $\ell_j$ tend to $\infty$ as $j\rightarrow +\infty$. Since we assume an ergodic base flow we have a constant index $k$ and thus omit it from now on, that is $\Lambda_{\chi,\ell}=\mathcal{K}(\ell,\chi)$ using the notation in \S\ref{discrete-time}.
We are in conditions to present the statement of \cite[Main Lemma pp. 154]{Kat} but for the flow context.

\begin{theorem}(Katok's shadowing lemma for flows)\label{KSL}
Fixed any $j\geq 1$ (thus $\chi_j>0$ and $\ell_j>1$), there exist positive numbers $K$, $\tau$, $\rho$ and $T$, such that given $\delta>0$, there exists $\epsilon=\epsilon(d,j,\delta)>0$ where the following holds: if for a given $z\in \Lambda_j$ and  $\hat\pi>0$ we have $X^{\hat\pi}(z)\in\Lambda_j$ and also $d(z,X^{\hat\pi}(z))<\epsilon$, then there exists $p=p(z)\in M$ such that:
\begin{enumerate}
\item (closing) $p$ is closed of period $\pi\in(\hat\pi-T,\hat\pi+T)$, i.e. $X^\pi(p)=p$;
\item (shadowing) $d(X^t(p),X^t(z))<\delta$ for all $t\in[0,\pi]$;
\item (hyperbolicity) $p$ is hyperbolic for the linear Poincar\'e flow $P_X^\pi(p)$;
\item (uniform hyperbolicity) the eigenvalues $\alpha$ of $P_X^\pi(p)$ satisfy $|\log|\alpha||>\pi\tau$;
\item (stable manifold) for all $t>0$ and $x,y\in W^s_{loc}(p)$ we have $d(X^t(x),X^t(y))<K e^{-\tau\,t}d(x,y)$;
\item (unstable manifold) for all $t>0$ and $x,y\in W^u_{loc}(p)$ we have $d(X^{-t}(x),X^{-t}(y))<K e^{-\tau\,t}d(x,y)$;
\item (uniform size) both $W^s_{loc}(p)$ and $W^u_{loc}(p)$ have size larger than $\rho$ and
\item (transversality) for all points $w\in \Lambda_{j}$ in a $\rho$-neighborhood of $z$, there exist small $t_w,s_w\in\mathbb{R}$, such that we have that $W^s_{loc}(p)$ intersects $W^u_{loc}(X^{t_w}(w))$ at exactly one point and  $W^u_{loc}(p)$ intersects $W^s_{loc}(X^{s_w}(w))$ at exactly one point.
\end{enumerate}
\end{theorem}

Next result is the continuous-time version of Proposition 4.7 from \cite{BeVar}. Let us assume that $\mu(M_0)>0$ where $M_0:=\{x\in M\colon \lambda^+(H,\mu)=0\}$.

\begin{proposition}\label{KatokViana}
Given $\hat\epsilon>0$ and $k\ge 2$ there exists a holonomy block $\tilde{\mathcal O}$ for $H$ so that 
$\mu(M_0\setminus \tilde \cO)<\hat\epsilon$, distinct dominated periodic points $\{p_i\}_{i=1}^k$ in $\tilde{\mathcal{O}}$ and a Hamiltonian 
linear differential system $\tilde H\in C^{r,\nu}(M, \mathfrak{sp}(2\ell,\mathbb K))$ such that
the following properties hold:
\begin{enumerate}
\item $W_{\text{loc}}^{u}(p_i)\pitchfork W_{\text{loc}}^{s}(p_{i+1})\not= \emptyset$ consists of one point for all $1\le i \le k$;
\item $p_i\in \supp(\mu\mid \tilde{\mathcal{O}} \cap X^{-\pi_i}(\tilde{\mathcal{O}}))$, where $\pi_i$ denotes the period of $p_i$; 
\item $\|A-B\|_{r,\nu}<\hat\epsilon$;
\item the Lyapunov spectrum of $B^{\pi_i}(p_i)$ is real and simple.
\end{enumerate}  
Finally, the set of cocycles $B$ satisfying (1),(2) and (4) is open in the $C^{r,\nu}$-topology.
\end{proposition}

\begin{proof}
The strategy to obtain (1) and (2) is modeled in \cite[Proposition 4.5]{Viana} and strongly uses Theorem~\ref{KSL}. We recall the highlights of Viana's proof borrowing the arguments in \cite[\S4.2]{Viana}. We will divide the proof
in small steps for the reader's convenience.

\vspace{.15cm}
\noindent \emph{Step 1:} Given $j$ (i.e. $\chi_j$ and $\ell_j$) such that $\mu(M_0\setminus\Lambda_j)<\hat\epsilon/2$, we fix $K$, $\tau$, $\rho$ and $T$ as in Theorem~\ref{KSL}. Let $\theta>0$ be such that $3\theta<\tau $. The Lemma~\ref{2.4} assures that for $\mu$-a.e. $x\in M_0$ there exists $T>0$ such that $x\in\mathscr{D}_H(T,\theta)$. Choose $T$ large enough so that $\mu(M_0\setminus \mathscr{D}_H(T,\theta))<\hat\epsilon/2$. We take a holonomy block defined by $\mathcal{O}=\Lambda_j\cap \mathscr{D}_H(T,\theta)$ such that $\mu(M_0\setminus \mathcal{O})<\hat\epsilon$ and $\mu(\mathcal{O})>0$;

\vspace{.15cm}
\noindent \emph{Step 2:} Fixing $\epsilon>0$, we find $k$ distinct points $\{z_i\}_{i=1}^k\subset M$ and $\{\pi_i\}_{i=1}^k\subset\mathbb{R}$ such that $z_i$ and $X^{\pi_i}(z_i)$ are in $B(x,\rho/2)$, $d(z_i,X^{\pi_i}(z_i))<\epsilon$ and $z_i \in \supp(\mu|\mathcal{O}\cap X^{-\pi_i}(\mathcal{O}))$. We may assume also that all $z_i$'s are at a distance larger than a fixed $r>0$. Now, we are in condition to apply Theorem~\ref{KSL} and complete part (1) of the lemma;

\vspace{.15cm}
\noindent \emph{Step 3:} Given $\delta=r/2$, there exists $\epsilon=\epsilon(d,j,\delta)>0$ given by Theorem~\ref{KSL} such that when feeding Step 2 with this $\epsilon$ the following holds: if for a given $z_i\in \mathcal O$ and  $\hat\pi_i>0$ we have $X^{\hat\pi_i}(z_i)\in\mathcal O$ and also $d(z_i,X^{\hat\pi_i}(z_i))<\epsilon$, then there exists distinct $p_i=p_i(z_i)\in M$ such that:
\begin{enumerate}
\item $p_i$ is closed of period $\pi_i\in(\hat\pi_i-T,\hat\pi_i+T)$, i.e. $X^{\pi_i}(p_i)=p_i$;
\item $d(X^t(p_i),X^t(z_i))<\delta$ for all $t\in[0,\pi_i]$;
\item $p_i$ is hyperbolic for the linear Poincar\'e flow $P_X^{\pi_i}(p_i)$;
\item the eigenvalues $\alpha$ of $P_X^{\pi_i}(p_i)$ satisfy $|\log|\alpha||>\pi_i\tau$;
\item for all $t>0$ and $x,y\in W^s_{loc}(p_i)$ we have $d(X^t(x),X^t(y))<K e^{-\tau\,t}d(x,y)$;
\item for all $t>0$ and $x,y\in W^u_{loc}(p_i)$ we have $d(X^{-t}(x),X^{-t}(y))<K e^{-\tau\,t}d(x,y)$;
\item both $W^s_{loc}(p_i)$ and $W^u_{loc}(p_i)$ have size larger than $\rho$ and
\item for all points $w\in \mathcal O$ in a $\rho$-neighborhood of $z_i$ there exist small $t_w,s_w\in\mathbb{R}$, such that we have that $W^s_{loc}(p)$ intersects $W^u_{loc}(X^{t_w}(w))$ at exactly one point and  $W^u_{loc}(p)$ intersects $W^s_{loc}(X^{s_w}(w))$ at exactly one point.
\end{enumerate}

\vspace{.15cm}
\noindent \emph{Step 4:} Now we will prove part (2) of the lemma. We start to define the subset $\tilde{\mathcal O}$. By Step 2 we have $z_i \in \supp(\mu|\mathcal{O}\cap X^{-\pi_i}(\mathcal{O}))$, so define a compact set $\mathcal O_i$ such that $\mathcal O_i\subset B(z_i,\nu)\cap \mathcal O$ and $X^{\pi_i}(\mathcal O_i)\subset B(X^{\pi_i}(z_i),\nu)\cap \mathcal O$ for some very small $\nu>0$. Using the transversality given in (h) of Step 3 we obtain that, for any $w\in \mathcal O_i$, there exist $t_w,s_w$ such that $W^s_{loc}(p_i)$ intersects $W^u_{loc}(X^{t_w}(w))$ at exactly one point and $W^u_{loc}(p_i)$ intersects $W^s_{loc}(X^{\pi_i+s_w}(w))$. Let $\Gamma^s_i\subset W^s_{loc}(p_i)$, respectively $\Gamma^u_i\subset W^u_{loc}(p_i)$, stand for those intersections. Now, for all $k,l\in\mathbb{N}$, we define $\Gamma_i^u(k)=X^{-\pi_i\,k}(\Gamma_i^u)$ and $\Gamma_i^s(l)=X^{\pi_i\,l}(\Gamma_i^s)$. A standard $\lambda$-lemma argument assures that, for any $k,l$, the local stable manifolds of points in $\Gamma^u_i(k)$ intersects in a transversal uniform way the local unstable manifold set of points in $\Gamma^s_i(l)$. Let $\mathcal O_{i}(k,l)$ denote that intersection. We fatten $\mathcal O_{i}(k,l)$ by a time $t$ where $t:=\frac{1}{2}\min\{t_w,s_w\}$ and $w\in\mathcal O_i$ (for all $i$) and we let $\hat{\mathcal O}_{i}(k,l):=\cup_{s\in[-t,t]}X^s(\mathcal O_{i}(k,l))$. Finally, we define the set $\tilde{\mathcal O}$ by:
$$\tilde{\mathcal {O}}:=\mathcal O\bigcup_{k+l\geq 1}\hat{\mathcal{O}}_{i}(k,l).$$
Observe that $\mu(M_0\setminus \tilde{\mathcal O})<\epsilon$. Figure 1 of \cite{Viana} is a nice illustration of what is happening inside a Poincar\'e section of the closed orbit $p_i$.

\vspace{.15cm}
\noindent \emph{Step 5:}  We obtain that $\tilde{\mathcal O}$ displays uniform hyperbolic rates. More precisely, that there exists $K'>K$ such that points $x,y$ in the local invariant manifolds of any $\xi\in\tilde{\mathcal O}$ satisfy the inequalities: $d(X^t(x),X^t(y))<K' e^{-\tau\,t}d(x,y)$ and  $d(X^{-t}(x),X^{-t}(y))<K' e^{-\tau\,t}d(x,y)$. The key ingredient is the continuity of the invariant manifolds (see ~\cite[Lemma 4.9]{Viana}) and the stability, on small segments of orbits, of the local product structure.

\vspace{.15cm}
\noindent \emph{Step 6:} Now we show that $\tilde{\mathcal O}$ is still a holonomy block. That is, there exists $\theta'>\theta$ (but such that $3\theta'<\tau$) such that $\tilde{\mathcal O}\subset \mathscr{D}_H(T,\theta')$. The arguments are like the ones in ~\cite[Lemma 4.10]{Viana} and we leave the details to the reader.

\vspace{.15cm}
\noindent \emph{Step 7:} Finally, we just have to prove that $p_i\in \supp(\mu\mid \tilde{\mathcal{O}} \cap X^{-\pi_i}(\tilde{\mathcal{O}}))$. We observe that for any $k,l>0$ we have
\begin{equation}\label{kl}
X^{\pi_i}(\mathcal O_i(k,l-1))=\mathcal O_i(k-1,l).
\end{equation}
We claim that, for all $k+l\geq 1$, $\mu(\hat{\mathcal{O}}_i(k,l))>0$. It is sufficient to show that $\mu^u\times\mu^s(\mathcal O_i(k,l))>0$ where this measure was treated in Definition~\ref{def:lpsf} and this can be achieved by borrowing the arguments in \cite[Lemma 4.11]{Viana}. We get that $p_i$ is accumulated by sets $\mathcal{O}_i(k,l)$, thus by sets $\hat{\mathcal{O}}_i(k,l)$, and which are inside $\tilde{O}$. Using (\ref{kl}) we get that the sets $\mathcal{O}_i(k,l)$ are also inside $X^{-\pi_i}(\mathcal O\cup_{k+l\geq 1}\mathcal{O}_{i}(k,l))$. Therefore, $p_i\in \supp(\mu\mid \tilde{\mathcal{O}} \cap X^{-\pi_i}(\tilde{\mathcal{O}}))$ and (2) is proved.

\vspace{.15cm}
\noindent \emph{Step 8:} In order to obtain (3) and (4) we proceed as in the proof of Proposition 4.7 from \cite{BeVar} but using the Lemma~\ref{perturb2} in \S\ref{PH} which is the Hamiltonian perturbation tool which allows us to perform the continuous-time perturbation in the vein of the one in the proof of Proposition 4.7 from \cite{BeVar}.

\end{proof}

\subsection{Proof of Theorem~\ref{thm:flow} for general flows} \label{sec:perturb.general}

The strategy follows the same steps as in \S\ref{sec:time.continuous.suspension} (continuous-time case with suspension flow in the base). There are essentially three main novelty key points: 
\begin{itemize}
\item obtaining the closed orbits which was performed in \S\ref{lot};
\item the reduction of the study of hyperbolicity on the normal cross sections cf. \S\ref{nh} and
\item the using of the induced return cocycle described in \S\ref{ic}. 
\end{itemize}

We begin by using the construction developed in \S\ref{lot} in order to obtain a large quantity of closed orbits near $x \in \supp(\mu\mid \cO)$ where $\cO$ is a positive $\mu$-measure holonomy block. Of course that those closed orbits can be seen as closed orbits associated to the Poincar\'e map $\mathcal P_X^t$ in a cross section $\Sigma$ and very near from $x$.

Then, since those closed orbits are hyperbolic we have large leaves $\cF^u$ and $\cF^s$. Hence, we can use the $\lambda$-lemma and build horseshoes and thus, a symbolic dynamics obtaining closed orbits with very large period which can be turn, via the perturbation Lemma~\ref{perturb2}, into closed orbits with real and simple spectrum cf. Proposition~\ref{KatokViana} 
(see also the final part of the proof in Proposition 4.7 from \cite{BeVar}).

Finally, the usual type of perturbation is done to break the holonomy. We use the definition of domination (in (\ref{ic3})) with respect to the induced cocycle $\Psi_H$ defined in (\ref{ic2}). Moreover, we act with the holonomies along the foliations $\cF^s$ and $\cF^u$ cf. Proposition~\ref{holono}. The perturbation is carried out using Lemma~\ref{perturb2} and intend to spoil the action of the induced cocycle $\Psi_H$ defined in (\ref{ic2}) and it is quite similar to the one performed in the Section~\ref{sec:time.continuous.suspension}. We observe that the perturbation of the Hamiltonian linear differential system is done in a small flowbox tubular neighborhood $\mathcal T$ of a given heteroclinic point $z$, forward asymptotic with the closed orbit $p$ and backward asymptotic with the closed orbit $q$ (recall \S\ref{sec:perturb.suspension}). For this reason our perturbation of the stable holonomy cannot interfere with the unstable holonomy which remain with the same action because $\cup_{t>0}X^t(\mathcal T)$ is far from the backward iterates of $\mathcal T$.

 \vspace{0,5cm}
 
\subsection*{Acknowledgements} 
PV was partially supported by CNPq and FAPESB. PV is also grateful to  CMUP - University of Porto for supporting his visit.

\vspace{0,5cm}


\end{document}